\documentclass[ngerman,twoside,numbers=noenddot]{scrartcl}
\KOMAoptions{DIV=0, 
            BCOR=0mm, 
            paper=a4, 
            fontsize=11pt, 
            titlepage=true}
\usepackage{authblk}
\usepackage[headsepline,automark,pagestyleset=KOMA-Script,markcase=ignoreuppercase]{scrlayer-scrpage}
\clearmainofpairofpagestyles
\ohead{\pagemark}
\ihead{\headmark}
\usepackage[left=3cm,right=3cm,top=1.5cm,bottom=3cm,head=35.9889pt,includehead]{geometry}
\usepackage[utf8]{inputenc} 
\usepackage[english]{babel} 
\usepackage[expansion=true, protrusion=true]{microtype}
\usepackage{amsmath}
\numberwithin{equation}{section}
\usepackage{calc}
\usepackage{mathtools}
\usepackage{amsfonts}
\usepackage{amssymb} 
\usepackage{amsthm}
\usepackage{upgreek}
\usepackage{bbold}
\usepackage{enumitem}
\usepackage[dvipsnames]{xcolor}
\usepackage{units}
\usepackage{graphicx} 
\usepackage[hypcap]{caption} 
\usepackage{booktabs} 
\usepackage{flafter} 
\usepackage[section]{placeins}
\usepackage{tikz-cd}
\usepackage{slashed}
\usepackage{multicol}
\usepackage{multirow}
\usepackage{float}
\usepackage{hyperref} 
\hypersetup{colorlinks=true, breaklinks=true, citecolor=red, linkcolor=blue, menucolor=red, urlcolor=cyan, bookmarksopen=false, bookmarksopenlevel=0, plainpages=false,hypertexnames=false}
\newtheorem{theorem}{Theorem}[section]
\newtheorem{proposition}[theorem]{Proposition}
\newtheorem{lemma}[theorem]{Lemma}

\theoremstyle{remark}
\newtheorem{remark}[theorem]{Remark}

\theoremstyle{definition}
\newtheorem{definition}[theorem]{Definition}

\usepackage[T1]{fontenc}
\usepackage{lmodern}
\usepackage{pdfpages}

\usepackage[autostyle]{csquotes}
\usepackage[numbers]{natbib}

\newcommand{\norm}[1]{\left\lVert#1\right\rVert}

\newcommand\restr[2]{{
  \left.
  #1
  \right|_{#2} 
  }}

\newcommand{\divergence}{\operatorname{div}}
\newcommand{\grad}{\operatorname{grad}}
\newcommand{\extd}[0]{\mathrm{d}}
\newcommand{\id}[0]{\mathrm{id}}

\newcommand{\idon}[1]{\id_{#1}}

\newcommand{\ric}{\mathrm{Rc}}
\newcommand{\richat}{\mathrm{\hat{R}c}}
\newcommand{\rictilde}{\mathrm{\Tilde{R}c}}
\newcommand{\rscal}{\mathrm{Sc}}

\newcommand{\genmet}{\mathcal{G}}

\newcommand{\genscal}{\mathcal{S}}
\newcommand{\genric}{\mathcal{R}c}

\newcommand{\scalbrack}[1]{\left\langle #1 \right\rangle}

\newcommand{\End}[0]{\mathrm{End}}
\newcommand{\Hom}[0]{\mathrm{Hom}}
\newcommand{\Sym}{\mathrm{Sym}}

\newcommand{\reals}{\mathbb{R}}
\newcommand{\absolute}[1]{\left\lvert #1 \right\rvert}

\newcommand{\tr}{\mathrm{tr}\:}
\newcommand{\trwith}[1]{\mathrm{tr}_{#1}\:}

\newcommand{\sym}{\mathrm{sym}}
\newcommand{\antisym}{\mathrm{antisym}}

\newcommand{\closure}[2][3]{%
{}\mkern#1mu\overline{\mkern-#1mu#2}}

\begin{document}
\pagenumbering{Roman}
\thispagestyle{empty}
\cleardoublepage
\setcounter{page}{1}
\thispagestyle{empty}

\cleardoublepage
\pagenumbering{arabic}

\title{The Initial Value Problem for \\ the Generalised Einstein Equations}

\date{\today}

\author{Oskar Schiller}

\maketitle

\tableofcontents
\pagenumbering{arabic}
\section{Abstract}

We discuss the initial value problem for the Einstein equations in Hitchin's generalised geometry for the case of closed divergence (which correspond to the equations of motion in the bosonic part of the NS-NS sector in type II ten-dimensional supergravity) and establish the existence of a maximal globally hyperbolic development (MGHD). The dynamical fields, defined on a manifold of dimension $n+1$, are the space-time metric, a scalar field known as the dilaton function, and a two-form known as the $B$-field. We develop a generalisation of the Lorenz gauge which, applied to the $B$-field (and combined with a suitable gauge condition breaking diffeomorphism invariance), renders the system a wave equation with principal symbol given by the (dynamical) metric. Given initial data, we construct a development satisfying the gauge conditions. We show that all other developments are (in the appropriate sense) related to this development by a diffeomorphism, establishing geometric uniqueness. The existence of the MGHD follows then by a famous result by Choquet-Bruhat and Geroch. 

In showing existence and geometric uniqueness of developments, we follow an approach developed in detail by Ringström for the Einstein equations coupled to a scalar field. In a preliminary section, we present a formulation which is disentangled from the specific assumptions made on the matter, so that adaptation to other systems is straightforward. 

\medskip
\textit{MSc classification}: 35Q76 (Einstein's equations); 53D18 (Generalized geometry a la Hitchin).\

\textit{Key words:} Einstein equations, generalized geometry, initial value problem, supergravity

\section{Introduction}

In generalised geometry, the generalised Einstein equations (GEE) are the natural analogue of the Einstein equations as they are known in conventional semi-Riemannian geometry. It is therefore natural to investigate if they yield a well-defined initial value problem (IVP). Moreover, it was established in \cite{waldramsupergravity} that the GEE (in the special case of \textit{exact} divergence) are equivalent to the bosonic part of the NS-NS sector in type II ten-dimensional supergravity, providing additional motivation to the formulation as an IVP.

The paper is structured as follows. In Section \ref{preliminaries_section}, we give preliminary results generally useful for establishing well-posedness of IVPs for Einstein-matter systems. These results were (in some cases implicitly) presented by Ringström in \cite{cauchy}, a work that (among other things) is concerned with well-posedness of the IVP for the Einstein equations coupled to a scalar field. We adopt its approach in our study of the IVP for the GEE.

In Section \ref{geneinsteineqs_section}, we present the GEE as one obtains them from generalised geometry. In the case of closed divergence, they are the following system of equations for a Lorentzian metric $g$, a closed three-form $H$, and a closed one-form $\xi$ called dilaton one-form (or simply dilaton), cf.\ (\ref{combinedgeneinsteineqs}):
\begin{equation*}
    \extd^*H = - i_{\xi} H, \qquad\quad \mathrm{Rc} = \frac{H^2}{4} - \nabla \xi, \qquad\quad \frac{\absolute{H}^2}{6} = \extd^*\xi + \absolute{\xi}^2.
\end{equation*}
Note that we refer to this system also as the \textit{string frame} GEE (as opposed to the Einstein frame GEE). We often work with a two-form potential $B$ for $H$, and a scalar potential $\phi$ for $\xi$. Closed divergence is also referred to as closed dilaton, as it implies $\extd \xi = 0$. 

In Section \ref{ivp_formulation_section}, we explain how to formulate the GEE as an IVP, following \cite{cauchy} in spirit. We introduce the following notion of initial data (cf.\ Definition \ref{initialdata_stringframe_def}).
\begin{definition}\label{initialdata_stringframe_def_intro}
    \textit{Initial data (in terms of fields)} for the (string frame) GEE in $n+1$ dimensions is a tuple $(\Sigma,g_0,k,H^\parallel_0, h_0, \xi_0^\parallel, x_0)$ consisting of an $n$-dimensional Riemannian manifold $(\Sigma,g_\Sigma)$, and the following objects on $\Sigma$: 
    \begin{enumerate}[label = (\arabic*)]
        \item a symmetric $(0,2)$-tensor $k$,
        \item a closed three-form $H^\parallel_0$ and a two-form $h_0$,
        \item a closed one-form $\xi^\parallel_0$ and a real-valued function $x_0$.
    \end{enumerate} 
    all such that the \textit{constraint equations} (\ref{genconstrainteqs}) are satisfied.
\end{definition}
The constraint equations get their name from their constraining the space of initial data; they are necessarily fulfilled by initial data coming from a development under the GEE. A posteriori, we see that they are in fact sufficient for the initial data to admit a development (cf.\ Theorem \ref{genEinsteinlocalexistencethm_intro}).

Ringström's approach is the following. First, embed the initial hypersurface in a suitable ambient space. Next, show that for every point on the hypersurface and in a gauge providing suitable uniqueness properties, one can construct on an open coordinate neighbourhood of the ambient space a globally hyperbolic development of the initial data. Finally, patch the local developments together to a development of the entire initial hypersurface. Geometric uniqueness can be established by local comparison to this development, as it was constructed in a gauge that endows it with a certain uniqueness. 

To obtain the local developments on coordinate patches, one requires a result such as \cite[Corollary 9.16]{cauchy} (which we include for completeness in this work as Theorem \ref{uniqueexistencehyperbolicpdethm}). It guarantees the existence of a unique maximal solution to quasi-linear hyperbolic PDE-systems for vector-valued functions $u\colon \reals^{n+1} \to \reals^{n+1}$ of the form
\begin{equation}\label{hyperbolicpdeeq_intro}
     g[u]^{\mu\nu} \partial_\mu \partial_\nu u = f[u]
\end{equation}
given compactly supported initial data. Note that $f$ and $g$ are members of specific classes of functions with arguments consisting of the point $(t,x) \in \reals^{n+1}$, and the solution $u$ and its first derivatives at that point; one employs the notation $g[u](t,x) = g(t,x,u(t,x),\partial_0u(t,x),...,\partial_nu(t,x))$ and similarly for $f$.

\begin{remark}
    As far as the author of this text is aware, \cite{cauchy} is the first work to provide a comprehensive and rigorous formulation of this approach to establish well-posedness.\footnote{It seems to the author of this work that it was well-known that such a strategy could be made to work. A remark to that extent is made in \cite[section 5.4.]{friedrichrendall2000}, and Ringström himself presents his foundational work as expository in the introduction to \cite{cauchy}.} For this reason, and because it is precisely the goal to establish basic well-posedness results for the system studied, we view it as desirable to follow Ringström's work. Its adaptation to other systems of matter is, at least in principle, straightforward. If one wants to discuss any given matter coupled to the Einstein equations, it should suffice to check that, in a local coordinate neighbourhood, the system is of a form such that a local existence and uniqueness result (such as Theorem \ref{uniqueexistencehyperbolicpdethm}) applies. The entire rest of the proof is then expected to hold still.\footnote{This adaptability should also be a feature of other approaches. For example, a claim to that extent is made in \cite[section 5.1.]{friedrichrendall2000}, and DeTurck seemingly refers to this being the case in \cite[remarks below Theorem 5.5]{deturck1983cauchy}.} However, this is not obvious, as \cite{cauchy} focuses on the particular case of the Einstein equations coupled to a scalar field and we consider a different matter model. It is furthermore complicated by the GEE only being of a good form after a locally defined conformal transformation (to the Einstein frame), which means that the background metric which was defined globally in \cite{cauchy} is in our setting only defined locally. 
\end{remark}

In Sections \ref{einsteinframe_section} and \ref{gaugeconditions_section}, we show how to locally relate the GEE to a system that is in coordinates of the form (\ref{hyperbolicpdeeq_intro}). This involves a conformal transformation of the metric, referred to as \textit{adopting the Einstein frame}, and then the construction of a system which is in coordinates of the form (\ref{hyperbolicpdeeq_intro}).

The Einstein frame is explained in Section \ref{einsteinframe_section}. The conformal transformation of the metric is given by  $\tilde{g} \coloneqq e^{-2 \kappa \phi} g$, where $\kappa = \frac{1}{d-2}$ and $d = n+1$ is the spacetime dimension. Note that the conformal factor is defined with respect to a choice of potential $\phi$ for the dilaton $\xi$. Thus, in general, there is no global choice of Einstein frame, and any two choices of Einstein frame are (over a simply connected set) related by a conformal transformation with constant conformal factor. The conformally transformed GEE are called the \textit{Einstein frame GEE}. They read
\begin{equation}\label{einsteinframecombinedeinsteineqs_intro}
    \begin{split}
        \Tilde{\extd}^*H &= -\frac{4}{d-2} \tilde{\iota}_\xi H \\
        \rictilde &= \frac{1}{d-2} \left[\xi \otimes \xi - \frac{e^{-4\kappa \phi}}{6}\absolute{H}_{\tilde{g}}^2\Tilde{g}\right]+ e^{-4\kappa\phi}\frac{H^{2,\tilde{g}}}{4}  \\ 
        \tilde{\Box}\phi &= -\frac{e^{-4\kappa\phi}}{6} \absolute{H}^2_{\tilde{g}}  
    \end{split} 
\end{equation}
Objects with a tilde employ in their definition the Einstein frame metric $\tilde{g}$. (For a precise explanation of the notation, cf.\ Proposition \ref{einsteinframeprop}.)

In Section \ref{gaugeconditions_section}, we perform a hyperbolic reduction of the Einstein frame GEE and construct a system which is in coordinates of the form (\ref{hyperbolicpdeeq_intro}). Crucially, the constructed system is equivalent to the Einstein frame GEE provided that DeTurck's gauge condition and the generalised Lorenz gauge condition - a new gauge condition which we develop - are satisfied. Note that there and in the following, we drop the tilde on the Einstein frame variables to avoid an overly cumbersome notation.

\textit{DeTurck's gauge condition} breaks diffeomorphism invariance by comparing the metric to a background metric $\Bar{g}$ \cite{deturck1983cauchy}. In contrast to the wave gauge, it is a geometric condition that does not depend on a special choice of coordinates. It can be stated as 
\[ 0 = \mathcal{D}_\mu = g_{\mu\nu}g^{\alpha\beta}(\Gamma_{\alpha\beta}^\nu - \Bar{\Gamma}_{\alpha\beta}^\nu),\]
where $\Gamma^\nu_{\alpha\beta}$ and $\Bar{\Gamma}^\nu_{\alpha\beta}$ respectively denote the Christoffel symbols of $g$ and $\Bar{g}$ in arbitrary coordinates (note that the expression yields a well-defined covector). With DeTurck's gauge implemented, one can replace in local expressions of the Ricci tensor particular appearances of the metric by corresponding appearances of the background metric, turning the Ricci tensor into the hyperbolic operator $\richat$. More precisely, $\richat = \ric + \nabla_{(\mu} \mathcal{D}_{\nu)}$. We remark that DeTurck's gauge condition and the wave gauge can both be understood as a special case of the \enquote{wave gauge with source functions}, as developed by Friedrich in \cite{friedrich1985hyperbolicity} and presented in a more modern context in \cite{friedrichrendall2000}. 

The \textit{generalised Lorenz gauge} is a geometric generalisation of the Lorenz gauge to higher forms, drawing on ideas from DeTurck's gauge condition. Remember that the (usual) Lorenz gauge is a condition imposed on the Maxwell-field $A$ (the one-form potential for the field strength tensor in Einstein-Maxwell theory). It can be stated as $\nabla^\mu A_\mu = 0$, and it can be employed to turn the Einstein-Maxwell system into a wave equation with principal symbol given by the metric \cite[§10.1.]{bruhatGRBook}. The generalised Lorenz gauge is a condition imposed on the $B$-field (the two-form potential for the dynamical part of $H$ with respect to a given background field $\Bar{H} \in \Omega_{\mathrm{cl}}^3(M)$, i.e.\ $H = \Bar{H} + \extd B$). It can be stated as 
\[ 0 = \big[\extd\extd^*_{g,\Bar{g}}B\big]_{\mu\nu} = - g^{\lambda\kappa}\nabla_{[\mu}\Bar{\nabla}_{|\lambda} B_{\kappa|\nu]} = 0 \]
with $\Bar{\nabla}$ the Levi-Civita connection of the background metric $\Bar{g}$, and the index notation $[\mu|\lambda\kappa|\nu]$ means antisymmetrisation in $\mu$ and $\nu$. It forces the action on $B$ of the modified Hodge wave operator $\hat{\Box}_{\mathrm{Hd}} = -\extd^*\extd - \extd \extd^*_{g,\Bar{g}}$ to agree with the co-exact wave operator $-\extd^*\extd$. We compare the usual and the generalised Lorenz gauge in Remark \ref{lorenzgauge_comparison_rem} and argue that the generalised Lorenz gauge is a natural generalisation of the Lorenz gauge to higher forms.
\begin{remark}
    We remark on the work \cite{choquetbruhatsugra} of Choquet-Bruhat establishing well-posedness of the Cauchy problem for 11-dimensional $N=1$ supergravity, which is a system related but also inequivalent to the one we study. Notably, with the so-called \enquote{three-index photon} $A$,  it contains an analogue of the $B$-field (which could accordingly be referred to as the two-index photon). In \cite{choquetbruhatsugra}, the coordinate-dependent generalisation $\partial^\lambda A_{\lambda\mu\nu} = 0$ of the Lorenz gauge is imposed on $A$ in local harmonic coordinates. 
\end{remark}

With the modified operators $\richat$ and $\hat{\Box}_{\mathrm{Hd}}$ described above, we obtain the modified system (cf.\ (\ref{modifiedcombinedeinsteineqs}))
\begin{equation}\label{modifiedcombinedeinsteineqs_intro}
    \begin{split}
        \hat{\Box}_{\mathrm{Hd}}B &= \extd^*\Bar{H} +\frac{4}{d-2} i_\xi H, \\
        \richat &= \frac{1}{d-2} \left[\xi \otimes \xi - \frac{e^{-4\kappa \phi}}{6}\absolute{H}_{g}^2\: g \right]+ e^{-4\kappa \phi}\frac{H^{2,g}}{4},  \\ 
        \Box\phi &= -\frac{e^{-4\kappa \phi}}{6} \absolute{H}^2_{g}.  
    \end{split}
\end{equation}
Herein, $H = \Bar{H} + \extd B$ with closed three-forms $H$ and $\Bar{H}$ and a two-form $B$. In Proposition \ref{localformprop}, we show that the principal symbol of this system is given by the (inverse) metric.
\begin{table}[H]
    \centering
    \begin{tabular}{|c|c|c|c|}
         \hline
         Variant of the GEE & String Frame & Einstein Frame & Modified Einstein Frame  \\
         \hline
         Relevance & study subject & auxiliary & metric principal symbol\\
         \hline
         Dynamical Objects & $g,H,\xi$ & $\tilde{g},H,\phi$ & $\tilde{g}, B, \phi$ \\
         \hline
         \multirow{2}{*}{Defined} & \multirow{2}{*}{globally} & locally w.r.t. & locally w.r.t.  \\
         & & a choice of $\phi$ &  a choice of  $\phi$, $\Bar{g}$, and $\Bar{H}$ \\
         \hline
         Relation to  & \multirow{2}{*}{identity  } & \multirow{2}{*}{$\tilde{g} = \exp(-2\kappa\phi) g$ } & equivalent to Einstein frame  \\
         String Frame & & & if gauges are implemented \\
         \hline
    \end{tabular}
    \caption{This table showcases the three versions of the GEE relevant to this work and their main properties. By convention, $g$ denotes the string frame metric, $\tilde{g}$ the Einstein frame metric, $H$ a closed three-form, $\xi$ a closed one-form (the dilaton), $\phi$ a real-valued function (a dilaton potential), $\Bar{g}$ a \enquote{background} metric, and $\Bar{H}$ a closed \enquote{background} three-form.}
    \label{tab:placeholder}
\end{table}
In Section \ref{intialdatasection}, we set up initial values for the dynamical fields $g$, $B$, and $\phi$ in the modified system (\ref{modifiedcombinedeinsteineqs_intro}). Note that the formal notion of initial data from Definition \ref{initialdata_stringframe_def_intro} does not determine the restrictions $\restr{g}{\Sigma}$ and $\restr{B}{\Sigma}$ to the initial hypersurface completely - excluded from the formal notion are components of $\restr{g}{\Sigma}$ and $\restr{B}{\Sigma}$ which are not invariant under diffeomorphisms or $B$-field transformations on $M$ which act as the identity on $T\Sigma \oplus T^*\Sigma$, i.e. gauge dependent components. Taking inspiration from \cite{cauchy}, we find in Lemmas \ref{initialdata_setup_metric_phi_lemma} and \ref{initialdata_setup_bfield_lemma} geometric conditions which characterise the initial values of $g$ and $B$ uniquely. Moreover, we prove in Lemma \ref{dilaton_gaugetrafo_lemma} that the conditions from Lemmas \ref{initialdata_setup_metric_phi_lemma} and \ref{initialdata_setup_bfield_lemma} as well as DeTurck's gauge condition and the generalised Lorenz gauge condition are invariant under a change of Einstein frame. This is crucial, because for a non-exact dilaton, there is no global choice of Einstein frame and thus developments defined in different Einstein frames have to be compatible.

In Section \ref{existence_ghd_section}, we prove existence of a string frame development for arbitrary string frame initial data. We prove with Lemmas \ref{localexistencelemma} and \ref{uniquesoldomainlemma} local (in space) existence and uniqueness results for the string frame GEE. The uniqueness result utilises that the local string frame developments come from gauge-fixed Einstein frame developments. Following the patching procedure of \cite{cauchy}, we obtain one of the main results of this work (cf.\ Theorem \ref{genEinsteinlocalexistencethm}):
\begin{theorem}\label{genEinsteinlocalexistencethm_intro}
    Let $(\Sigma, g_0, k, H_0, h_0, \xi_0, x_0)$ be initial data for the string frame GEE (\ref{combinedgeneinsteineqs}). Then there exists a globally hyperbolic string frame development of the data.
\end{theorem}

In Section \ref{uniqueness_ghd_section}, we prove the geometric uniqueness of globally hyperbolic developments. That is, we prove that any two developments are extensions of a common development. We achieve this by showing in Proposition \ref{combinedgauge_prop} that, given any Einstein frame development $(M,\tilde{g},H,\phi)$ of initial data, one can locally implement via a gauge transformation DeTurck's gauge, the generalised Lorenz gauge, and the geometric initial conditions from Lemmas \ref{initialdata_setup_metric_phi_lemma} and \ref{initialdata_setup_bfield_lemma}. This allows for a comparison of $(M,\tilde{g},H,\phi)$ with the local Einstein frame developments that enter the construction of the string frame development in Theorem \ref{genEinsteinlocalexistencethm_intro}. Following again \cite{cauchy}, we obtain another main result of this work (cf.\ Theorem \ref{genEinsteinlocaluniquenessthm}):
\begin{theorem}\label{genEinsteinlocaluniquenessthm_intro}
    Let $(\Sigma, g_0, k, H_0, h_0, \xi_0, x_0)$ be initial data for the string frame equations (\ref{combinedgeneinsteineqs}). Denote by $(D, g, H, \xi)$ a globally hyperbolic development constructed as in Theorem \ref{genEinsteinlocalexistencethm}, $D\subset M = \reals \times \Sigma$. Assume that we have another development $(M', g', H', \xi')$ with embedding $i \colon \Sigma \hookrightarrow M'$. 
    
    Then, there is a tubular neighbourhood $\tilde{D} \subset D$ of $\Sigma$ and a smooth time orientation preserving diffeomorphism $\psi \colon \tilde{D} \to \psi(\tilde{D}) \subset M'$ with $ \restr{\psi}{\Sigma} = i$ relating the two solutions, i.e.\ $\psi^*g' = g$, $\psi^* H' = H$, and $\psi^* \xi' = \xi$.
\end{theorem}

Finally, in Section \ref{mghd_section}, we present the famous result by Choquet-Bruhat and Geroch \cite{choquet1969global} that establishes the existence of an MGHD for arbitrary Einstein-matter systems, assuming they admit local developments of initial data and solutions are geometrically unique. From it and Theorems \ref{genEinsteinlocalexistencethm_intro} and \ref{genEinsteinlocaluniquenessthm_intro} we get the existence of an MGHD (cf.\ Theorem \ref{mghd_thm}).
\begin{theorem}
    Let $\mathcal{I}$ be initial data for the string frame GEE. Then there exists a string frame MGHD of $\mathcal{I}$. It is unique up to diffeomorphism.
\end{theorem}

\paragraph{The Cauchy Problem for General Einstein-Matter Systems.}
One can understand the GEE with closed divergence as an Einstein-matter system with matter given by $H$ and $\xi$. Existence and uniqueness of solutions to Einstein-matter equations has been the subject of extensive investigation, and is established under varying assumptions of well-behavedness on the matter. In the following, we give a brief overview of relevant literature, outlining ways of obtaining well-posedness in different approaches, and motivating our own in the Sections \ref{einsteinframe_section}-\ref{uniqueness_ghd_section}.

Generally, the assumption of \enquote{well-behavedness} on the matter is needed to guarantee two properties. First, the energy-momentum tensor has to be divergence-free as a consequence of the matter equations alone (i.e.\ without requiring the energy-momentum tensor to equal the Einstein tensor). This assumption ensures that the gauge condition of choice for turning the Einstein equations into a hyperbolic system (e.g.\ DeTurck's gauge) is preserved under development of initial data by the Einstein equations, cf.\ Lemma \ref{metricgaugefield_sourcefreewaveeq_lemma_preliminaries}. We see in Section \ref{einsteinframe_section} that the GEE only satisfy this requirement after a conformal transformation of the system (referred to in physics as a transformation to the Einstein frame). The Einstein frame requires a choice of potential $\phi$ for the dilaton, $\xi = \extd \phi$, and can therefore only be implemented locally. Second, in a suitable gauge, the combined Einstein-matter system has to form a hyperbolic system with (in some notion) locally well-behaved IVP. Given these two properties, there are well-understood procedures to construct a development to initial data, prove that any two developments are an extension of a common development, and thus conclude the existence of an MGHD. 

One of the most prominent treatments of the Cauchy problem for Einstein-matter systems may be found in Hawking and Ellis \cite[section 7.7]{hawkingellis}. We will not pursue this approach, but would like to explain its main idea and difficulties. Fundamentally, Hawking and Ellis establish the existence and uniqueness of solutions via a fixed-point argument (based on the seminal work \cite{bruhat1952seminalwork} by Choquet-Bruhat). Given initial data on a hypersurface in an ambient spacetime, the argument involves the map $\alpha \colon g_0 \mapsto g_1$ on the space of Lorentzian metrics, where $g_1$ is defined as follows. Solve the matter equations with $g_0$ as a background, then find a solution $h$ to the linearised Einstein equations with given stress-energy tensor, and finally set $g_1 = g_0 + h$. Hawking and Ellis prove $\alpha$ to be a contraction, yielding the unique solution to the Einstein-matter equations. 

The approach by Hawking and Ellis produces as main requirements on the matter equations that they have a well-defined Cauchy problem \textit{given a background metric}, and that they have good stability behaviour under changes of the background metric. Both properties should be straightforward to show in our setting. What complicates things is the additional requirement that the stress-energy tensor is polynomial in the matter fields, their first covariant derivatives, and the metric. From Proposition \ref{emtensorprop}, we know this requirement to be violated by the energy-momentum tensor associated to $H$ and $\phi$ in the Einstein frame. However, this seems to be a purely technical point. A more substantial problem in adopting this approach here is that it works with the equations on the whole ambient spacetime, but we only know the GEE to be well-behaved in the locally defined Einstein frame. Therefore, we abandon the approach presented in \cite{hawkingellis} in favour of Ringström's approach as presented in \cite{cauchy}.

We briefly note that one can also obtain well-posedness results in the ADM formalism. A comprehensive treatment in this formalism  of the Cauchy problem for general Einstein-matter systems  can be found in \cite{fischer1979initial}. 

\paragraph{Acknowledgements.} I am sincerely grateful to Hans Ringström for insightful comments he provided on his work. My gratitude also extends to the mathematics department of UNED Madrid for granting me a doctoral visiting position. Especially during my visit, but also beyond, I have greatly benefited from the support of and discussions with Miguel Pino Carmona and Carlos Shahbazi, who have obtained similar results in the analytic category but for arbitrary signature \cite{bunk2025cauchyproblemgradientgeneralized}. Finally, I am indebted to Vicente Cortés for his many valuable comments, his guidance, and his overall support.

\subsection{Notation and Conventions}

$(M,g)$ denotes a $d = n+1$-dimensional Lorentzian manifold. We set $\kappa = \frac{1}{d-2}$. We call a time-oriented Lorentzian manifold a \textit{spacetime}. We call a tuple $(M,g,H,\phi)$ consisting of a spacetime $(M,g)$, a closed three-form $H \in \Omega^3(M)$ and a function $\phi \in C^\infty(M)$ a \textit{SuGra spacetime}. We denote the Beltrami wave operator as $\Box = \nabla^\mu \nabla_\mu$ and the Hodge wave operator as $\Box_{\mathrm{Hd}} = - \extd^* \extd - \extd \extd^*$. We denote the interior product with a vector field $X$ by $i_X$, so that for a one-form $\xi$ we have $i_X(\xi) = \xi(X)$. For convenience, we also write $i_\xi$ for the interior product with the vector field $\xi^\sharp$ obtained from raising an index on a given one-form $\xi$.

Working in indices, we employ the Einstein summation convention throughout. Greek indices $\alpha,\beta,\mu,\nu,...$ take values $0,...,n$, where $0$ is assumed to correspond to a timelike direction. Latin indices $a,b,i,j,...$ take values $1,...,n$ and are assumed to correspond to spacelike directions.

We make use of round brackets to denote normalised symmetrisation, and square brackets to denote normalised antisymmetrisation. Thus, for example, the Lie derivative of the metric can be written as $L_X g_{\mu\nu} = 2 \nabla_{(\mu} X_{\nu)}$. Working without indices, we denote the same operations by a superscript \enquote{$\sym$} and \enquote{$\antisym$}, respectively. Thus, $L_X g = 2 [\nabla X^\flat]^\sym$ with $\flat\colon T^*M \to TM$ the musical isomorphism coming from $g$.

Given a semi-Riemannian hypersurface $\Sigma$ in $(M,g)$ with unit normal $N$, we decompose the restriction of a $p$-form $A \in \Omega^p(M)$ in terms of its tangent and normal part. The tangent part $A^\parallel \in \Omega^p(\Sigma)$ is the restriction of $A$ to $T\Sigma$, and the normal part $A^\perp \in \Omega^{p-1}(\Sigma)$ is the restriction of $i_N A$ to $T\Sigma$. In particular, $A = A^\parallel + \varepsilon\: N^\flat \wedge A^\perp$, where $N^\flat = g(N,\cdot)$ and $\varepsilon = g(N,N)$.

In this setting, it is sometimes of interest to decompose the restricted exterior derivative $\restr{\extd A}{\Sigma}$ in terms of contributions from the parallel and normal parts of $A$ and $\nabla_N A$. Denoting by $k$ the second fundamental form on $\Sigma$, one finds the following contribution (cf.\ Lemma \ref{extdhypersurflemma}) 
\begin{equation*}
    \left(k \cdot A^\parallel\right)_{i_1...i_p} \coloneqq (-1)^{p} p\: k_{[i_1}^j A^\parallel_{i_2...i_p]j}
\end{equation*}
Herein, $\cdot$ is the natural Lie algebra action of endomorphisms on the space of $p$-forms (for convenience, we use the letter $k$ also for the Weingarten map, i.e.\ the endomorphism obtained from the second fundamental form by raising an index).
\section{Preliminaries: Techniques Applicable to IVPs for Einstein-Matter Systems}\label{preliminaries_section}

We present here those parts of the work \cite{cauchy} by Ringström which are relevant to any Einstein-matter system. While all results presented in this section are heavily based on \cite{cauchy}, most of them are only established implicitly in \cite{cauchy}. For detailed proofs, we refer to the upcoming PhD thesis auf the author of this work \cite{oskar_phdthesis}.

\subsection{Quasi-Linear Hyperbolic PDEs with Metric Principal Symbol}\label{existence_ghd_preliminaries}

In this section, we state a result on the existence and uniqueness of solutions to quasi-linear hyperbolic PDEs with principal symbol given by the inverse of a Lorentzian metric which is part of the dynamics, provided that the non-linearity satisfies a minor well-behavedness assumption. In an appropriate gauge, many important Einstein-matter systems are of this type, cf.\ Remark \ref{einstein_matter_local_form_rem}.

Let $M = \reals \times \Sigma$ for a smooth manifold $\Sigma$. Denote by $L(M) \subset \Sym^2(TM)$ the bundle of non-degenerate symmetric two-tensors with Lorentzian signature, by $E'$ a vector subbundle of the tensor bundle $\mathcal{T}M$, and set $E = L(M) \times E'$. Take an arbitrary point $p \in \Sigma$, and a coordinate neighbourhood $(U_\Sigma,x^m)$ of $p$, and define coordinates $x^\mu = (x^0 = t, x^m)$ on $U = \reals \times U_\Sigma$. Then, we study on $U$ the PDE
\begin{equation}\label{localexistenceproof_hyperbolicpdeeq}
    g^{\mu\nu}\partial_\mu \partial_\nu u - f[u] = 0, \qquad u = (g,u') \in \Gamma(E|_U)
\end{equation}
where $f \in C^\infty (L_n \times \reals^{N' + (n+1) N + n+1}, \reals^N)$ with $N' = \operatorname{rank} E'$, $N = N' + (n+1)^2$, and $L_n$ the space of $n+1$-dimensional Lorentz matrices. We employed the notation
\begin{equation*}
    f[u](t,x) \coloneqq f(t,x,g_{\mu\nu},u_j',\partial_\mu u_j).
\end{equation*}
We finally make the assumption that
\begin{equation*}
    f[u](t,x) = F\big(g_{\mu\nu}, u'_j, \partial_\mu u_j, \partial^\alpha B_l(t,x)\big)
\end{equation*}
for some \enquote{background fields} $B \in C^\infty(M, \reals^K)$, $K \in \mathbb{N}_0$, and a smooth $\reals^N$-valued function $F$ depending on variables as indicated ($\alpha \in \mathbb{N}_0^{n+1}$ represents multi-indices up to degree $k \in \mathbb{N}_0$, $\absolute{\alpha} \leq k$) such that $F(g_{\mu\nu},0,...,0) = 0$.

The following proposition provides a generalisation of a result implicitly established in \cite[proof of Theorem 14.2]{cauchy}. For a proof, cf.\ \cite[Proposition 5.1.]{oskar_phdthesis}.
\begin{proposition}\label{localexistenceprop_preliminaries}
    Let $U_0,U_1 \in C^\infty(U_\Sigma, E)$ be initial data for the equation (\ref{localexistenceproof_hyperbolicpdeeq}) such that, writing $U_0 = (g_0, U'_0)$, $(g_0)_{\mu\nu}$ takes values in the space of canonical Lorentz matrices $\mathcal{C}_n$.  
    
    Then, for every open neighbourhood $O \subset U$ of $p$, there exists an open neighbourhood $W \subset O$ of $p$ on which the equation (\ref{localexistenceproof_hyperbolicpdeeq}) admits a globally hyperbolic development $u = (g,u')$ of the initial data on $W \cap \Sigma$ given by restriction of $U_0$ and $U_1$ (i.e.\ $u = (g,u')$ solves (\ref{localexistenceproof_hyperbolicpdeeq}), $(W,g)$ is globally hyperbolic, and we have $\restr{U_0}{W \cap \Sigma} = \restr{u}{W\cap \Sigma}$ as well as $\restr{U_1}{W\cap \Sigma} = \restr{\partial_\mu u}{W\cap\Sigma}$). $W$ can be chosen such that the component-matrix $(g_{\mu\nu})_{\mu\nu}$ takes values in the space of canonical Lorentz matrices $\mathcal{C}_n$ and, in particular, such that $\grad t$ is timelike on $W$.
\end{proposition}

\paragraph{Uniqueness of Local Solutions.} We want to discuss the uniqueness of local solutions to equations of the type (\ref{hyperbolicpdeeq}), as guaranteed to exist by Proposition \ref{localexistenceprop_preliminaries}.

We focus on the following setting. Let 
\[f \in C^\infty (L_n \times \reals^{N' + (n+1)N + n+1},\reals^N),\]
where $N' \in \mathbb{N}_0$ and $N = N' + (n+1)^2$. Denote $\Sigma = \{0\}\times \reals^n \subset \mathbb{R}^{n+1}$. For $i=1,2$, let $W_i \subset O_i \subset \reals^{n+1}$ be open subsets such that $W_1 \cap W_2 \cap \Sigma \neq \emptyset$, and denote $\Sigma_i \coloneqq W_i \cap \Sigma$. Let $g_i \colon O_i \to \mathcal{C}_n$ be a Lorentzian metric and $u'_i \colon O_i \to \reals^{N'}$, $i=1,2$, be such that over $W_i$, $u_i= (g_i,u'_i)$ solves
\begin{equation*}
    g_i^{\mu\nu}\partial_\mu \partial_\nu u_i = f[u_i]
\end{equation*}
Assume furthermore that $O_i$ is convex with respect to $g_i$, and that $\Sigma_i$ is spacelike Cauchy in $(W_i,g_i)$. 

The following Proposition establishes local uniqueness in the desired sense. It is based on an argument found in \cite[proof of Theorem 14.2]{cauchy}. For a proof, cf.\ \cite[Proposition 5.5.]{oskar_phdthesis}.
\begin{proposition}
    \label{uniquesoldomainprop_preliminaries}
    Assume that $\closure{W}_i$ is compact and contained in $O_i$, $i=1,2$. Assume that $u_1 = u_2$ and $\grad u_1 = \grad u_2$ on $\Sigma_1 \cap \Sigma_2$. Then $u_1 = u_2$ on the whole intersection $W_1 \cap W_2$.
\end{proposition}

\subsection{DeTurck's Gauge Condition}\label{deturck_gauge_preliminaries}
In this section, we recall that the DeTurck gauge propagates well, and that one can locally implement it by acting with a diffeomorphism. We are interested in DeTurck's gauge condition, because it casts the Einstein vacuum equations a system of the form of quasi-linear hyperbolic PDEs with principal symbol given by the metric, as discussed in section~\ref{existence_ghd_preliminaries}.

\paragraph{The DeTurck gauge propagates well.} 
Let $\mathcal{D} \in \Gamma(T^*M)$ (later given by (\ref{ddefeq})), and assume that the modified Ricci tensor $\richat_{\mu\nu} = \ric_{\mu\nu}+\nabla_{(\mu}\mathcal{D}_{\nu)}$ (later as in (\ref{richateq})) satisfies the following \enquote{Einstein equations}:
\[ \richat - \frac{\trwith{g}\richat}{2}\: g = T \]
Herein, $T$ is any divergence-free two-tensor. We refer to $\mathcal{D} = 0$ as \textit{DeTurck's gauge condition}.

In this setting, an argument based on \cite[chapter 14.1]{cauchy} shows that the initial vanishing of $\mathcal{D}$ and $\nabla \mathcal{D}$ on a subset $\Omega$ of an initial hypersurface $\Sigma$ implies that $\mathcal{D}$ vanishes on the Cauchy development $D(\Omega)$, cf.\ also \cite[Lemma 5.6.]{oskar_phdthesis}.
\begin{lemma}\label{metricgaugefield_sourcefreewaveeq_lemma_preliminaries}
    Let $(M,g)$ be a globally hyperbolic Lorentzian manifold with spacelike Cauchy hypersurface $\Sigma$. Let $\bar{g}$ be another Lorentzian metric on $M$, the background metric. Assume that $\richat - \frac{1}{2}\trwith{g}\richat = T$, where $T$ a $(0,2)$-tensor such that $\divergence T = 0$. Assume that on $\Omega \subset \Sigma$, $\mathcal{D} = 0$ and $\nabla \mathcal{D} =0$. Then $\mathcal{D} = 0$ on the entire Cauchy development $D(\Omega)$.
\end{lemma}

\paragraph{Implementing the DeTurck gauge.} The existence of a diffeomorphism implementing the DeTurck gauge is a classical result \cite[section 2]{friedrich1985hyperbolicity}, cf.\ also \cite[proof of Theorem 14.3]{cauchy}.

Given a Lorentzian manifold $(M,g)$ and a background metric $\bar{g}$, DeTurck's gauge condition is
\begin{equation*}
    0 = \mathcal{D}_\mu[g,\bar{g}] = g_{\mu\nu}g^{\alpha\beta}(\Gamma^\nu_{\alpha\beta} - \bar{\Gamma}^\nu_{\alpha\beta}).
\end{equation*}
Note that the right-hand side is coordinate-independent and thus yields a well-defined covector $\mathcal{D}$. If it is clear which metrics are being referred to, we also write $\mathcal{D} = \mathcal{D}[g] = \mathcal{D}[g,\bar{g}]$. We discuss this gauge condition in Section \ref{gaugeconditions_section} and in particular in Remark \ref{deturck_gauge_rem}.

The following result is implicitly established in \cite[proof of Theorem 14.3]{cauchy}. For a proof, cf.\ \cite[Proposition 5.9.]{oskar_phdthesis}.
\begin{proposition}\label{deturckgauge_prop_preliminaries}
    Let $(M,\bar{g})$ and $(M',g)$ be smooth spacetimes, and $\Sigma$ a smooth spacelike hypersurface $M$ that is also embedded in $M'$ with embedding $i\colon\Sigma \hookrightarrow M'$. Denote by $N$ and $N'$ the respective future-directed unit normal on $\Sigma$ and $i(\Sigma)$. Then, for all $p \in i(\Sigma)$ there exist neighbourhoods $p \in W' \subset M'$ and $i^{-1}(p) \in W \subset M$ and a diffeomorphism $f \colon W \to W'$ such that
    \begin{enumerate}[label=(\roman*)]
        \item $f^*g$ is in DeTurck gauge with respect to $\bar{g}$, i.e $\mathcal{D}[f^*g] = 0$, and
        \item $f_*N_q = N'_{i(q)}$ for all $q \in W_\Sigma \coloneqq W \cap \Sigma$, and
        \item $\restr{f_*}{TW_\Sigma}= i_*$.
    \end{enumerate}
\end{proposition}

\subsection{Unions and Intersections of Globally Hyperbolic Subsets}\label{gh_subsets_preliminaries} 
We prove here two small auxiliary results employed in the patching together of local solutions (Theorem \ref{genEinsteinlocalexistencethm}) and local gauge transformations (Theorem \ref{genEinsteinlocaluniquenessthm}).

The following Lemma is implicitly established in \cite[proof of Theorem 14.2]{cauchy}. For a proof, cf.\ \cite[Lemma 5.10.]{oskar_phdthesis}.
\begin{lemma}\label{gh_union_lemma_preliminaries}
    Let $(M,g)$ be a Lorentzian manifold, $\Sigma \subset M$ a spacelike hypersurface, and $t \colon M \to \reals$ a smooth function such that $\grad t$ is timelike everywhere and $\Sigma = t^{-1}(0)$. For all $i$ in some index set $I$, let $W_i \subset M$ be an open subset such that $(W_i, \restr{g}{W_i})$ is globally hyperbolic with Cauchy hypersurface $W_i \cap \Sigma$. Then, $D= \bigcup_i W_i$ is globally hyperbolic with Cauchy hypersurface $D \cap \Sigma$.
\end{lemma}

A result similar to the following is established in \cite[proof of Theorem 14.3]{cauchy}. For a proof, cf.\ \cite[Lemma 5.11.]{oskar_phdthesis}.
\begin{lemma}\label{intersection_gh_lemma_preliminaries}
    Let $(M,g)$ be globally hyperbolic with Cauchy hypersurface $\Sigma$. Let $U$ and $V$ be globally hyperbolic subsets with $U\cap V \neq \emptyset$ such that $U_\Sigma = U \cap \Sigma$ and $V_\Sigma = V \cap \Sigma$ are Cauchy hypersurfaces in $U$ and $V$, respectively. 
    
    Then $U\cap V$ is globally hyperbolic with Cauchy hypersurface $U_\Sigma \cap V_\Sigma$. In particular, an isometry $f\colon (U\cap V,g) \to (M',g')$ is uniquely determined by the restriction of its differential to $U_\Sigma \cap V_\Sigma$.
\end{lemma}

\section{The Generalised Einstein Equations}\label{geneinsteineqs_section}

In this section, we discuss the generalised Einstein equations (GEE). In this work, we do not focus on their description in generalised geometry and instead state the equations in terms of objects naturally living on the manifold $M$. The interested reader is referred to \cite{fernandez, genricciflow,streets2024genricciflow, vicenteoskar, vicentethomas}.

Let $(M,g)$ be a Lorentzian manifold, $H \in \Omega^3_{\mathrm{cl}}(M)$ a closed three-form, $X \in \Gamma(TM)$ a vector field and $\xi \in \Gamma(T^*M)$ a one-form. Denote by $H^2 \in \Sym^2(M)$ the symmetric two-tensor obtained from $H \otimes H$ by making two non-trivial contractions with the metric $g$, more precisely $H^2_{\mu\nu} = g^{\kappa \lambda} g^{\rho \pi} H_{\mu\kappa\rho} H_{\nu\lambda\pi}$.
\begin{definition}
    The tuple $(M,g,H,X,\xi)$ is called \textit{generalised Ricci flat} if \cite[Corollary 4.7]{vicentethomas}
    \begin{equation} \label{geneinsteineqs}
        \begin{split}
            4\:\mathrm{Rc} &= H^2 - 4 [\nabla\xi]^\sym, \qquad\quad \extd^*H = 2 [\nabla X^\flat]^\antisym - i_\xi H
        \end{split}
    \end{equation}
    Furthermore, $(M,g,H,X,\xi)$ is called \textit{generalised scalar flat} if \cite[Corollary 4.4]{vicentethomas}
    \begin{equation*}
        \rscal = \frac{\absolute{H}^2}{12} + 2 \extd^*\xi + \absolute{X}_g^2 + \absolute{\xi}_g^2.
    \end{equation*}
    Finally, $(M,g,H,X,\xi)$ is called \textit{generalised Einstein (in the string frame)} if it is both generalised Ricci and scalar flat. 
\end{definition}
\begin{remark}\label{divoprem}
    In generalised geometry, the GEE refer to equations for a pair $(\genmet, \divergence)$ consisting of a generalised metric $\genmet$ and a divergence operator $\divergence$ on an exact Courant algebroid. In our framework, the metric $g$ and the three-form $H$ encode the generalised metric, and the \enquote{generalised vector} $X + \xi$ encodes the divergence operator, cf.\ also \cite[Proposition 2.38. and §2.4]{genricciflow}. We note that we introduced a conventional factor in the definition of $X$ and $\xi$. Writing as explained in \cite{genricciflow} $\divergence = \divergence^\genmet - \scalbrack{e,\cdot}$, we set $e = 2 (X +\xi)$, whereas in \cite{genricciflow} the convention is $e = X + \xi$. This agrees with the conventions employed in \cite{vicentethomas}.
\end{remark}
\begin{remark}
    Note that generalised Ricci flatness does not in general imply generalised scalar flatness.
\end{remark}
\begin{remark}
    Compared to the condition of generalised Ricci flatness found in \cite[Proposition 3.30]{genricciflow}, our version does not include the compatibility equations $L_X g = 0$ and $\extd \xi = i_X H$. They ensure compatibility of the divergence operator and the generalised metric in the sense that they are equivalent to the generalised vector field $X +\xi$ being an infinitesimal isometry \cite[Lemma 3.32.]{genricciflow}. Note that the compatibility equations are implied by our assumption (\ref{closeddilatonconditioneq}). An in-depth account of the different notions of curvature employed in the generalised geometry literature has been given in \cite{rubiogenricequivalence}.
\end{remark}

For our discussion of the initial value problem, we restrict our attention to the case of 
\begin{equation}\label{closeddilatonconditioneq}
    X= 0 \qquad \text{and} \qquad \extd \xi = 0
\end{equation}
As mentioned in Remark \ref{divoprem}, in generalised geometry, the generalised vector $X + \xi \in \Gamma(TM \oplus T^*M)$ comes from a pair consisting of a generalised metric and a divergence operator. Whether or not (\ref{closeddilatonconditioneq}) is satisfied depends only on the choice of divergence operator (although the precise choice of $\xi$ depends on the generalised metric). Hence, we refer to condition (\ref{closeddilatonconditioneq}) also as \textit{the divergence operator being closed}\footnote{Note that this nomenclature is taken from \cite[Definition 3.40]{genricciflow}, where a pair consisting of a generalised metric and a divergence operator is called closed if it satisfies (\ref{closeddilatonconditioneq}).}. It is called exact, if $\xi = \extd \phi$ for a function $\phi \in C^\infty(M)$. The exact case is of particular importance for three reasons. 
\begin{enumerate}[label = (\arabic*)]
    \item The GEE for exact divergence can be derived from a variational principle. This is done by defining the generalised Einstein Hilbert action to be the integral of the generalised scalar curvature (with measure defined by the divergence operator), and then varying over the space of generalised metrics and exact divergence operators \cite[§ 3.7]{genricciflow}.
    \item In the case of exact divergence, $\xi = \extd\phi$ for some $\phi \in C^\infty(M)$, the GEE are known to correspond to the equations of motion for the bosonic part of the NS-NS sector in type II ten dimensional theories of supergravity \cite{waldramsupergravity}. Under this correspondence, the function $\phi$ takes on the role of the \textit{dilaton}, a physical field.
    \item The correspondence to supergravity allows us to borrow a trick well-known in the physics literature: going into the Einstein frame. This is a name for a conformal transformation of the metric with the dilaton $\phi$. The trick is crucial to our analysis, as only in the Einstein frame we see the PDE to be a wave equation of a form that we recognise and can deal with. We elaborate on this in Section \ref{einsteinframe_section}.
\end{enumerate}
The reasons (1) and (3) given above indicate why working in the slightly more general case of a closed dilaton $\xi$ is feasible. This is because in both cases, access to a potential $\xi =\extd \phi$ is only required locally. For reason (1), this is because it suffices to define a generalised Einstein Hilbert action locally to obtain the GEE (which are local equations) from their variation. For reason (3), it is because we will only employ the Einstein frame locally.

We want to rewrite the GEE slightly. By (\ref{geneinsteineqs}) we have
\begin{equation}\label{scalcruveq}
    \rscal = \frac{\absolute{H}^2}{4} + \extd^*\xi.
\end{equation}
Thus, $\genric = 0$ and $\genscal = 0$ combined are equivalent to
\begin{equation}\label{combinedgeneinsteineqs}
    \extd^*H = - i_{\xi} H, \qquad\quad \mathrm{Rc} = \frac{H^2}{4} - \nabla \xi, \qquad\quad \frac{\absolute{H}^2}{6} = \extd^*\xi + \absolute{\xi}^2.
\end{equation}
We refer to the first equation as \textit{the $H$-field equation}, to the second equation as \textit{the Ricci equation}, and to the third equation as \textit{the dilaton equation}. Writing $H = \Bar{H}+\extd B$ for a closed three-form $\Bar{H}$ and a two-form $B$, we also refer to the first equation as \textit{the $B$-field equation}. 
\section{The Initial Value Problem in for the Generalised Einstein Equations}

In solving the initial value problem, we closely follow \cite{cauchy}, which proves the existence of a maximal globally hyperbolic development (MGHD) in the case of the Einstein equations coupled to a scalar field. It should be noted that \cite{cauchy} also provides a comprehensive introduction to the required mathematical theory, and for that reason is the main source for the present paper.

The main ingredient to proving the existence of an MGHD in \cite{cauchy} is that locally, over a suitable coordinate patch, the Einstein equations become a system of PDEs for a vector valued function $u\colon \reals^{n+1} \to \reals^N$, $1\leq n\in\mathbb{N}$ that takes, after some modification, the form
\begin{equation}\label{hyperbolicpdeeq}
    \begin{split}
        g[u]^{\mu\nu} \partial_\mu \partial_\nu u &= f[u], \\
        u(T_0, \cdot ) &= U_0, \\
        \partial_t u(T_0, \cdot) &= U_1,
    \end{split}
\end{equation}
where $T_0 \in \reals$, $g \in C^\infty(\reals^{N+(n+1)N+n+1}, \mathcal{C}_n)$ a so-called \textit{$C^\infty$ ($N,n$)-admissible metric} ($\mathcal{C}_n$ denotes the space of canonical Lorentz matrices), $f\in C^\infty(\reals^{N+(n+1)N+n+1},\reals^N)$ a so-called \textit{$C^\infty$ ($N,n$)-admissible non-linearity}, and $U_0,U_1 \in C^\infty_0(\reals^n, \reals^N)$ initial data. We employed the notation
\begin{equation*}
    g[u](t,x) \coloneqq g(t,x,u(t,x),\partial_0u(t,x),...,\partial_nu(t,x))
\end{equation*}
and similarly for $f$. We provide the definitions of the two notions of $C^\infty$ $(N,n)$-admissibility in appendix \ref{nonlinear_waveeqs_appendix}. Of main importance to us is \cite[Corollary 9.16.]{cauchy} (which we include for completeness in this work as Theorem \ref{uniqueexistencehyperbolicpdethm}): The system of PDEs (\ref{hyperbolicpdeeq}) admits a unique maximal solution.

\begin{remark}\label{einstein_matter_local_form_rem}
    It is known that, locally and in the right gauge, the vacuum Einstein equations \cite[Theorem 7.1.]{bruhatGRBook}, the Einstein-Maxwell equations \cite[§10.1.]{bruhatGRBook}, and the Einstein equations coupled to a scalar field \cite[§14]{cauchy} are of the form (\ref{hyperbolicpdeeq}). A sizeable part of this work will be to prove that, in fact, the GEE are also locally of this form. 
\end{remark}

Evidently, the equations (\ref{hyperbolicpdeeq}) are of second order. To view the GEE as a second order system, we have to work with potentials $B$ and $\phi$ that respectively encode the dynamics of $H$ and $\xi$. Regarding the dilaton, we write locally $\xi = \extd \phi$. Regarding the three-form, we take another approach and introduce a \enquote{background field} $\bar{H}$ and split $ H = \bar{H} + \extd B$. Then we have $B$ as the dynamical object, meaning that its development will be defined by solving the equations of motion, while the background field $\Bar{H}$ encodes global data, i.e.\ the cohomology class $[H]$. Note that all developments we consider will be defined on a tubular neighbourhood $D$ of $\Sigma_0 \coloneqq \{0 \} \times \Sigma$ in $M=\reals \times \Sigma$. In particular, $\Sigma$ will be a deformation retract of the manifold $D$, implying that the initial value of $H$ on $\Sigma$ determines the cohomology class $[H]$. The approach allows for a simple geometric condition uniquely determining the initial value of the $B$-field, making it straightforward to obtain local uniqueness of solutions: Choosing the background field such that initially $\restr{\bar{H}}{\Sigma} = \restr{H}{\Sigma}$, one may demand $\restr{B}{\Sigma} = 0$. 

Implementing $H = \bar{H} + \extd B$ and $\xi = \extd \phi$ in the GEE (\ref{combinedgeneinsteineqs}), we obtain the system 
\begin{equation}\label{combinedgeneinsteineqs_inpotentials}
    \extd^*\extd B =  - \extd^*\bar{H} - i_{\xi} H, \qquad\quad \mathrm{Rc} = \frac{H^2}{4} - \nabla^2 \phi, \qquad\quad \Box \phi =  - \frac{\absolute{H}^2}{6} + \absolute{\xi}^2,
\end{equation}
where we used that $\Box \phi = \nabla^\mu \nabla_\mu \phi = - \extd^* \extd \phi$.

Note that, taken at its face, the system (\ref{combinedgeneinsteineqs_inpotentials}) is not of the form (\ref{hyperbolicpdeeq}). In the Ricci equation, there are two issues, and a third issue is in the $B$-field equation. 

The first issue arises from the second order term $\nabla^2\phi$. This causes the equation not to be locally of the desired form, since the vector-valued function $u$ has to encompass all dynamical objects $g_{\mu\nu}, B_{\mu\nu}$, and $\phi$, i.e. roughly\footnote{Of course, this would not yield a well-defined function on $\reals^{n+1}$, as one needs to make use of a coordinate chart to be able to speak of the components. But, for now, to discuss the shape of the PDEs, we will ignore this technical detail.}
\begin{equation*}
    u = (g_{\mu\nu}, B_{\mu\nu}, \phi)
\end{equation*}
This can be overcome by a conformal transformation of the metric: 
\begin{equation*}
    \Tilde{g} \coloneqq e^{-2\phi\: /(d-2)} g
\end{equation*}
Under this conformal transformation, the Ricci tensor absorbs the $\nabla^2\phi$-term. Considering then $(\Tilde{g}_{\mu\nu}, B_{\mu\nu}, \phi)$ as our dynamical variables, the issue is solved. In the physics literature, this is known as \enquote{adopting the Einstein frame}. The name is due to the fact that only in this \enquote{frame} energy-momentum is off-shell conserved, i.e.\ without imposing the Ricci-equation and only due to the dilaton and $B$-field equations one has
\begin{equation*}
    \divergence_{\Tilde{g}}T[\Tilde{g},B, \phi] = 0
\end{equation*}
where $T$ is the energy-momentum tensor. We discuss this in Section \ref{einsteinframe_section}.

The other two issues are similar: In the Ricci and $B$-field equation, the respective differential operator acting on the metric or $B$-field is not locally of the form $g^{\mu\nu}\partial_\mu \partial_\nu +  \text{lower order derivatives}$. DeTurck's gauge (among others) was developed precisely to solve this issue. Drawing from the DeTurck and Lorenz gauge, we introduce a gauge condition that achieves the same for the $B$-field equation. We describe both gauge conditions in Section \ref{gaugeconditions_section}.

\subsection{The Initial Value Formulation}\label{ivp_formulation_section}

Our goal is to show that the initial value problem to the GEE (\ref{combinedgeneinsteineqs}) is well-behaved. That is, we want to show that every choice of initial data for the equations uniquely admits a development. So, to begin with, we have to understand what constitutes initial data for the equations (\ref{combinedgeneinsteineqs}).

Initial data for the vacuum Einstein equations is well-understood (cf.\ e.g.\ \cite[§14.2.]{cauchy}). It consists of a Riemannian manifold $(\Sigma, g_\Sigma)$ and a symmetric two-tensor $k$ on $\Sigma$ (representing the second fundamental form), all such that the constraint equations are satisfied (cf.\ \cite[cf.\ Proposition~13.3]{cauchy})
\begin{equation*}
    \begin{split}
        \rscal_\Sigma + (\tr k)^2 - \absolute{k}^2 = 0 \\
        \divergence_\Sigma k - \nabla^\Sigma \tr k = 0
    \end{split}
\end{equation*}
Herein, $\nabla^\Sigma$ denotes the LC connection of $g_\Sigma$, $\rscal_\Sigma$ its scalar curvature, and $\divergence_\Sigma k$ the contraction of the first two indices of $\nabla^\Sigma k$ with $g_\Sigma$. A \textit{development} of this data is a semi-Riemannian embedding $\iota\colon (\Sigma, g_\Sigma) \hookrightarrow (M,g)$ into a Lorentzian manifold $(M,g)$ satisfying the Einstein equations. The development is called \textit{globally hyperbolic}, if $(M,g)$ is.

For the GEE, we have with $H$ and $\xi$ two additional fields, for which the initial data has to determine the initial value. Notice that the equation of motion for $H$ and $\xi$ is first order. Hence, we expect initial data for (\ref{combinedgeneinsteineqs}) to additionally include the following objects defined on $\Sigma$: a closed three-form $H_0^\parallel$, a two-form $h_0$, a closed one-form $\xi_0^\parallel$ and a real-valued function $x_0$. A \textit{development} of this data is then a Riemannian embedding $\iota\colon(\Sigma, g_\Sigma) \hookrightarrow (M,g)$ into a Lorentzian manifold $(M,g)$, together with a closed three-form $H \in \Omega^3_{\mathrm{cl}}(M)$ and a closed one-form $\xi \in \Omega^1_{\mathrm{cl}}(M)$, all such that the GEE are satisfied and on $\iota(\Sigma)$
\begin{equation}\label{initialdatacompatbilityeq}
    \iota^*H = H_0^\parallel, \qquad\quad \iota^*[H(N)] = h_0, \qquad\quad \iota^*\xi = \xi_0^\parallel , \qquad\quad \iota^*[\xi(N)] = x_0.
\end{equation}
In the formalism of generalised geometry, one can prove that for a tuple of objects $(\Sigma, g_\Sigma,k, H_0^\parallel, h_0, \xi_0^\parallel, x_0)$ as above to have a development, it is necessary that \cite[Corollary 5.22.]{vicenteoskar}
\begin{equation}\label{genconstrainteqs}
    \begin{split}
        \rscal_\Sigma + (\tr k)^2 - \absolute{k}^2 &= \frac{\absolute{H^\parallel_0}^2}{12} + \frac{\absolute{h_0}^2}{4}  +2 (\extd^\Sigma)^*\xi^\parallel_0 + 2 (\tr k) x_0  + \absolute{\xi^\parallel_0}^2-x^2_0 \\
        \divergence_{\Sigma} k - \extd^\Sigma \tr k &= \frac{1}{4}C(h_0, H^\parallel_0) - \extd^\Sigma x_0 + i_{\xi^\parallel_0} k \\
        0  &=\left((\extd^\Sigma)^* +i_{\xi^\parallel_0}\right) h_0 \\
    \end{split}
\end{equation}
Herein, we denoted by $C(h_0,H_0^\parallel)$ the one-form obtained from $h_0$ and $H_0^\parallel$ by contracting indices as follows:
\begin{equation*}
    C_i(h_0,H_0^\parallel) = g_\Sigma^{kl}g_\Sigma^{mn} (h_0)_{km} (H_0^\parallel)_{lni}
\end{equation*}
Equations (\ref{genconstrainteqs}) are the (string frame) constraint equations for the GEE. 
\begin{remark}
    Instead of obtaining these constraint equations from the generalised Gauß and Codazzi equations established in \cite{vicenteoskar}, one can also obtain them from a more down-to-earth viewpoint. In the GEE (\ref{combinedgeneinsteineqs}), the Ricci equation is known to produce constraint equations, 
    and these are the first two equations in (\ref{genconstrainteqs}). The $H$-field equation is a form wave equation, and we see these to produce a constraint equation in Theorem \ref{coexact_wave_eqs_main_thm}. This is the third equation in (\ref{genconstrainteqs}).
\end{remark}

We formalise the notions of initial data, constraint equations, and the IVP for the GEE in the following two definitions. These are based on the analogous definition for the Einstein equations coupled to a scalar field in \cite[Definition 14.1]{cauchy}. 
\begin{definition}\label{initialdata_stringframe_def}
    \textit{Initial data (in terms of fields)} for the (string frame) GEE (\ref{combinedgeneinsteineqs}) in $n+1$ dimensions is a tuple $(\Sigma,g_0,k,H^\parallel_0, h_0, \xi^\parallel_0, x_0)$ consisting of an $n$-dimensional Riemannian manifold $(\Sigma,g_\Sigma)$, and the following objects on $\Sigma$: 
    \begin{enumerate}[label = (\arabic*)]
        \item a symmetric $(0,2)$-tensor $k$,
        \item a closed three-form $H^\parallel_0$ and a two-form $h_0$,
        \item a closed one-form $\xi^\parallel_0$ and a real-valued function $x_0$.
    \end{enumerate} 
    all such that the \textit{constraint equations} (\ref{genconstrainteqs}) are satisfied.
\end{definition}
\begin{definition}\label{development_stringframe_def}
    Let $(\Sigma,g_0,k,H^\parallel_0, h_0, \phi_0, \phi_1)$ be initial data. Then the \textit{initial value problem} (IVP) is to find a \textit{(string frame) development} of the initial data, that is a SuGra spacetime $(M,g,H,\phi)$ satisfying the (string frame) GEE and a Riemannian embedding $i \colon (\Sigma,g_0) \hookrightarrow (M,g)$ such that $k$ corresponds to the second fundamental form on $i(\Sigma)$ and $H$ and $\phi$ induce on $\Sigma$ the given initial data via the relations (\ref{initialdatacompatbilityeq}). A development $(M,g,H,\phi)$ is called globally hyperbolic, if $i(\Sigma)$ is a Cauchy hypersurface in $(M,g)$. 
\end{definition}

Given a background three-form $\bar{H}$ and working locally, we can express the fields $H$ and $\xi$ via potentials $B$ and $\phi$:
\begin{equation}\label{potentialeq}
    H = \bar{H} + \extd B, \qquad \xi = \extd \phi.
\end{equation}
Note that the initial data for $\phi$ is described by $\phi_0 = \restr{\phi}{\Sigma}$ and $\phi_1 = N(\phi)$. It holds $\xi_0^\parallel = \extd^\Sigma\phi_0$ and $x_0 = \phi_1$. 

To discuss initial data in terms of the potential $B$, we consider the following
\begin{lemma}\label{hypersurfacepotfieldsrelationlemma}
    Let $H = \extd B$. Furthermore, express $B$ and $\nabla_N B$ in terms of objects natural on $\Sigma$:
    \begin{equation}\label{bfielddecompeq}
        \restr{B}{\Sigma} = B_0^\parallel - N^\flat \wedge b_0, \qquad\qquad \restr{\nabla_N B}{\Sigma} = B_1^\parallel - N^\flat \wedge b_1
    \end{equation}
    where $B_0^\parallel, B_1^\parallel \in \Omega^2(\Sigma)$ and $b_0, b_1 \in \Omega^1(\Sigma)$.
    Then, with $H_0^\parallel$ and $h_0$ as in (\ref{initialdatacompatbilityeq}),
    \begin{equation*}
        \begin{split}
            H_0^\parallel &= \extd^\Sigma B_0^\parallel, \qquad\quad h_0 = B_1^\parallel - \extd^\Sigma b_0 - k\cdot B_0^\parallel. \\
        \end{split}
    \end{equation*}
    Herein, $\cdot$ is the natural action of endomorphisms on differential forms.
\end{lemma}
\begin{proof}
    One finds
    \begin{equation*}
        \begin{split}
            \extd B &= \extd^\Sigma B_0^\parallel - N^\flat\wedge\left[B_1^\parallel-\extd^\Sigma b_0 -k \cdot B_0^\parallel \right]. 
        \end{split}
    \end{equation*}
    as a special case of Lemma \ref{extdhypersurflemma}.
\end{proof}
We summarise the interpretation of this statement in the following Remark.
\begin{remark}\label{initialdatadefrem}
    In the decomposition (\ref{bfielddecompeq}), only the initial values of $B_0^\parallel$ and the difference $B_1^\parallel - \extd^\Sigma b_0$ are constrained by the initial value of $H$. Note that, for initial data induced on a spacelike hypersurface one can always achieve $b_0=0$ via a closed $B$-field transformation on $M$, cf.\ Lemma \ref{bfieldgaugelemma}. With this assumption, the initial data precisely constrains $B_0^\parallel$ and $B_1^\parallel$ - the former up to addition of a closed two-form, and then the latter uniquely (but depending on that freedom).
\end{remark}
We conclude from the preceding discussion that \emph{initial data (partially) in terms of potentials} differs from initial data in terms of fields in that one replaces $(H_0^\parallel, h_0^\parallel)$ by $(B_0^\parallel, B_1^\parallel)$ and/or $(\xi^\parallel_0,x_0)$ by $(\phi_0,\phi_1)$, demanding then that with (\ref{potentialeq}) the constraint equations (\ref{genconstrainteqs}) are satisfied.

\subsection{The Einstein Frame}\label{einsteinframe_section}

Recall that the $\nabla^2\phi$-term in (\ref{combinedgeneinsteineqs_inpotentials}) produces out-of-place second order derivatives, if we want to understand the system as being of the form (\ref{hyperbolicpdeeq}). In this section, we discuss the GEE in the Einstein frame, and show that in this frame, there are no such out-of-place derivatives. Moreover, we see that in this frame, the energy-momentum tensor is divergence-free as a consequence of the $B$-field and dilaton equation alone.

We emphasize that the Einstein frame can (for non-exact dilaton) only be obtained locally, after writing $\xi = \extd \phi$. In the context of the IVP, one should therefore imagine the SuGra spacetimes we discuss here to be a suitable small subset of the development one considers.

The Einstein frame is achieved by the conformal transformation described in the following Proposition.
\begin{proposition}\label{einsteinframeprop}
    Let $(M,g,H,\phi)$ be a SuGra spacetime of dimension $d=n+1$. Denote  $\xi = \extd \phi$. Define with $\kappa = \frac{1}{d-2}$ the Einstein frame metric as $\tilde{g} \coloneqq e^{-2 \kappa \phi} g$. Then, combined generalised Ricci and scalar flatness (\ref{combinedgeneinsteineqs}) is equivalent to 
    \begin{equation} \label{einsteinframecombinedeinsteineqs}
        \begin{split}
            \Tilde{\extd}^*H &= -\frac{4}{d-2} \tilde{\iota}_\xi H \\
            \rictilde &= \frac{1}{d-2} \left[\xi \otimes \xi - \frac{e^{-4\kappa \phi}}{6}\absolute{H}_{\tilde{g}}^2\Tilde{g}\right]+ e^{-4\kappa\phi}\frac{H^{2,\tilde{g}}}{4}  \\ 
            \tilde{\Box}\phi &= -\frac{e^{-4\kappa\phi}}{6} \absolute{H}^2_{\tilde{g}}  
        \end{split} 
    \end{equation}
    Herein, $\tilde{\iota}_\alpha H = H(\alpha^{\tilde{\sharp}})$, where the musical ismorphism $\tilde{\sharp}$ comes from the Einstein frame metric $\tilde{g}$. Furthermore, $\tilde{\Box} = \tilde{\nabla}^\mu\tilde{\nabla}_\mu$, and $H^{2,\tilde{g}} \in \mathrm{Sym}^2(M)$ denotes the symmetric two-tensor obtained by contracting with $\tilde{g}$, i.e.\ $H^{2,\tilde{g}}_{\mu\nu} = \tilde{g}^{\kappa\lambda} \tilde{g}^{\omega\pi} H_{\mu\kappa\omega} H_{\nu\lambda\pi} = e^{4\kappa\phi}H^2_{\mu\nu}$.
\end{proposition}
\begin{proof}
    Equivalence between (\ref{combinedgeneinsteineqs}) and (\ref{einsteinframecombinedeinsteineqs}) follows from simple calculations, utilizing the following well-known identities for conformal transformations $\tilde{g} = e^{2\varphi} g$:
    \begin{equation*}
        \begin{split}
            \rictilde &= \ric - (d-2)(\nabla^2 \varphi - \extd\varphi \otimes \extd \varphi)- (\Box \varphi + (d-2) \absolute{\extd \varphi}_g^2) g, \\
            \tilde{\nabla} \alpha &= \nabla \alpha + \scalbrack{\alpha, \extd\varphi}_g g - (\extd\varphi \otimes \alpha + \alpha \otimes \extd \varphi), \\
            e^{2\varphi} \tilde{\extd}^*\omega &= \extd^*\omega - (d-2p) i_{\extd \varphi} \omega.
        \end{split}
    \end{equation*}
    Herein, $\alpha$ denotes an arbitrary one-form, and $\omega$ a $p$-form. Inserting $\varphi = -\kappa\phi$, this becomes
    \begin{equation*}
        \begin{split}
            \rictilde &= \ric + (\nabla^2 \phi + \frac{1}{d-2} \extd\phi \otimes \extd \phi)+ \frac{1}{d-2}(\Box \phi - \absolute{\extd \phi}_g^2) g, \\
            \tilde{\nabla} \alpha &= \nabla \alpha - \frac{1}{d-2} \big[ \scalbrack{\alpha, \extd\phi}_{\tilde{g}} \tilde{g} - (\extd\phi \otimes \alpha + \alpha \otimes \extd\phi) \big], \\
            \tilde{\extd}^*\omega &= e^{2 \kappa\phi}\extd^*\omega + \frac{d-2p}{d-2} \tilde{\iota}_{\extd \phi} \omega.
        \end{split}
    \end{equation*}
    We start with the first equation.
    \begin{equation*}
        \begin{split}
            &\extd^*H = - i_\xi H \\
            \Longleftrightarrow \qquad&  e^{-2\kappa\phi}\left(\tilde{\extd}^*H - \frac{d-6}{d-2}\tilde{\iota}_{\xi}H\right) = - i_\xi H \\
            \Longleftrightarrow \qquad&  \tilde{\extd}^*H = \frac{-4}{d-2}\tilde{\iota}_{\extd \phi} H  \\
        \end{split}
    \end{equation*}
    Noting that the dilaton equation can be stated as $\Box \phi - \absolute{\extd\phi}^2_g = - \frac{\absolute{H}^2}{6}$, we continue with the second equation.
    \begin{equation*}
        \begin{split}
            &4\: \ric = H^2 - 4 \nabla \xi \\
            \Longleftrightarrow \qquad &\rictilde = \frac{1}{d-2} \left[\extd \phi \otimes \extd \phi + \left( \Box \phi - \absolute{\extd \phi}^2_g\right)g\right] + \frac{H^2}{4}  \\
            \Longleftrightarrow \qquad &\rictilde = \frac{1}{d-2} \left[\extd \phi \otimes \extd \phi - \frac{\absolute{H}^2_g}{6}\: g\right] +  \frac{H^2}{4}  \\
            \Longleftrightarrow \qquad &\rictilde = \frac{1}{d-2} \left[\xi \otimes \xi - e^{-4\kappa\phi} \frac{\absolute{H}^2_{\tilde{g}}}{6}\: \Tilde{g}\right] +  e^{-4\kappa\phi}\frac{H^{2,\tilde{g}}}{4}    \\
        \end{split}
    \end{equation*}
    Finally, we calculate
    \begin{equation*}
        \begin{split}
            \tilde{\Box}\phi & = \trwith{\tilde{g}}(\tilde{\nabla}^2\phi )  = \trwith{\tilde{g}}(\tilde{\nabla}\extd \phi ) \\
            & = \trwith{\tilde{g}}\left(\nabla^2 \phi - \frac{1}{d-2} \big[ \absolute{\extd\phi}^2_{\tilde{g}}\tilde{g} - 2\: \extd\phi \otimes \extd \phi \big]\right) \\
            & =e^{2\kappa\phi}\left(\Box\phi - \absolute{\extd\phi}^2_{g}\right) \\
            & =- e^{-4\kappa\phi}\frac{\absolute{H}^2_{\tilde{g}}}{6}.
        \end{split}
    \end{equation*}
    We are done.
\end{proof}
Let us discuss the IVP formulation of the Einstein frame GEE. The notions of \textit{Einstein frame initial data} (in fields and potentials) and \textit{Einstein frame developments} are directly analogous to the corresponding notions in the string frame developed in Definitions \ref{initialdata_stringframe_def} and \ref{development_stringframe_def}, bearing in mind that one has to work with a potential for the dilaton for the Einstein frame to be defined. It remains to develop the Einstein frame constraint equations and the transformation that relates string frame and Einstein frame initial data. We start with the latter.
\begin{lemma}\label{initialdata_einsteinframe_lemma}
    Let $(M, g,H,\phi)$, $H= \bar{H}+ \extd B$, be a SuGra spacetime with spacelike hypersurface $\Sigma$ on which the initial data $(g_0,k,H_0,h_0,\phi_0,\phi_1)$ is induced. Denote by $N$ the future-pointing unit normal on $\Sigma$. Then, after the conformal transformation $\tilde{g} = e^{-2\kappa\phi} g$, $\kappa = \frac{1}{d-2}$, (corresponding to a transformation from the string to the Einstein frame), the SuGra spacetime $(\tilde{g}, H, \phi)$ induces the initial data $(\tilde{g}_0,\tilde{k},\tilde{H}_0,\tilde{h}_0,\tilde{\phi}_0,\tilde{\phi}_1)$ given by
    \begin{equation*}
       \begin{split}
           \tilde{g}_0 &= e^{-2\kappa\phi}g_0, \qquad\quad \tilde{k} = e^{-\kappa\phi} [k - \kappa N(\phi) g_0],\\
           \tilde{H}_0 &= H_0, \qquad\qquad \tilde{h}_0 = e^{\kappa\phi}h_0 \\
           \tilde{\phi}_0&= \phi_0,\qquad\qquad\tilde{\phi}_1 = e^{\kappa\phi} \phi_1
       \end{split}
    \end{equation*}
    Furthermore, the components $B_0,B_1,b_0$ and $b_1$ of a two-form $B \in \Omega^2(M)$ in the decomposition
    \begin{equation}
        \restr{B}{\Sigma} = B_0 - N^\flat \wedge b_0, \qquad\qquad \restr{\nabla_N B}{\Sigma} = B_1 - N^\flat \wedge b_1
    \end{equation}
    transform as
    \begin{equation*}
        \begin{split}
            \tilde{B}_0 &= B_0, \qquad\qquad \tilde{B}_1 = e^{\kappa\phi}[B_1 + 2 \kappa N(\phi) B_0-\kappa b_0 \wedge \extd^\Sigma\phi],\\
            \tilde{b}_0 &= e^{\kappa\phi} b_0, \qquad\quad \tilde{b}_1 = e^{2\kappa\phi}\left[b_1 - 2\kappa N(\phi) b_0 + \kappa i_{\extd^\Sigma\phi}B_0   \right].
        \end{split}
    \end{equation*}
\end{lemma}
\begin{proof}
    The formulas for $\tilde{g}_0, \tilde{H}_0$, and $\tilde{\phi}_0$ are obvious. Denote by $N$ and $\tilde{N}$ the future-pointing unit normal on $\Sigma$ with respect to $g$ and $\tilde{g}$, respectively. Then $\tilde{N} = e^{\kappa\phi}N$. The formulas for $\tilde{h}_0$ and $\tilde{\phi}_1$ follow trivially. Furthermore,
    \begin{equation*}
        \begin{split}
            \tilde{k} &= [\tilde{\nabla} \tilde{g}\tilde{N}]^\parallel = [\tilde{\nabla}(e^{-\kappa \phi}gN)]^\parallel = e^{-\kappa \phi} \restr{\left[\tilde{\nabla} gN -\kappa \extd\phi \otimes gN \right]}{T\Sigma} \\
            &= e^{-\kappa \phi}\restr{\left[\nabla gN -\kappa \scalbrack{gN,\extd\phi}_g g + \kappa (\extd\phi\otimes gN + gN \otimes \extd\phi)  -\kappa \extd\phi \otimes gN  \right]}{T\Sigma} \\
            &= e^{-\kappa \phi} \left[k - \kappa N(\phi)g_0 \right].
        \end{split}
    \end{equation*}
    Let us now discuss how the $B$-field components transform. Again, trivially $\tilde{B}_0 = B_0$ and $\tilde{b}_0 = e^{\kappa \phi}b_0$. Also,
    \begin{equation*}
        \begin{split}
            \tilde{\nabla}_{\tilde{N}}B &= e^{\kappa \phi}\tilde{\nabla}_N B =e^{\kappa \phi} \left[\nabla_N B + 2\kappa N(\phi) B - \kappa (i_N B) \wedge \extd \phi - \kappa (gN) \wedge (i_{\extd\phi} B)\right].\\
        \end{split}
    \end{equation*}
    It follows that
    \begin{equation*}
    \begin{split}
        \tilde{B}_1 &= \restr{\tilde{\nabla}_{\tilde{N}}B}{T\Sigma} = e^{\kappa \phi}[B_1 + 2\kappa N(\phi) B_0 -\kappa b_0 \wedge \extd^\Sigma\phi], \\
        \tilde{b}_1 &= [\tilde{\nabla}_{\tilde{N}} B](\tilde{N}) =e^{+2\kappa \phi}\left[b_1 + 3\kappa N(\phi) b_0 + \kappa i_{\extd\phi}B - \kappa b_0(\grad_g \phi) gN  \right] \\
        &=e^{2\kappa\phi}\left[b_1 + 2 \kappa N(\phi) b_0 + \kappa i_{\extd^\Sigma\phi}B_0 \right], \\
        \end{split}
    \end{equation*}
    as claimed.
\end{proof}
Next, we compute the Einstein frame energy-momentum tensor and check that it is indeed divergence free as a consequence of the $B$-field and dilaton equation alone. For the rest of this section, we drop the tilde over the conformally transformed metric for ease of notation. 
\begin{proposition}\label{emtensorprop}
    Let $(M,g,H,\phi), \xi = \extd \phi,$ be a SuGra spacetime satisfying the Einstein frame GEE (\ref{einsteinframecombinedeinsteineqs}). Then, the energy-momentum tensor $T = \ric - \frac{\rscal}{2} g$ is given by
    \begin{equation*}
        T = \frac{1}{d-2}\left[\xi \otimes \xi - \frac{\absolute{\xi}^2}{2} g\right]  + \frac{e^{-4\kappa \phi}}{4}\big[H^2  -\frac{\absolute{H}^2}{6} \: g\big]\\
    \end{equation*}
    Furthermore, assuming only the $H$-field and dilaton equation in (\ref{einsteinframecombinedeinsteineqs}), the energy-momentum tensor is divergence-free.
\end{proposition}
\begin{proof}
    We calculate
    \begin{equation*}
        \begin{split}
            T &= \frac{1}{d-2}\big[\xi \otimes \xi -\frac{1}{6} e^{-4\kappa \phi}\absolute{H}^2_g \: g\big]  + e^{-4\kappa \phi}\frac{H^2}{4} \\
            &\quad - \frac{1}{2(d-2)} g\: \tr\big[\xi \otimes \xi -\frac{1}{6} e^{-4\kappa \phi}\absolute{H}^2_g \: g\big] - \frac{e^{-4\kappa \phi}}{8}\: g\: \tr(H^2)\\
            &= \frac{1}{d-2}\big[\xi \otimes \xi -\frac{1}{6} e^{-4\kappa \phi}\absolute{H}^2_g \: g\big]  + e^{-4\kappa \phi} \frac{H^2}{4}  \\
            &\quad - \frac{1}{2(d-2)} g\: \big[\absolute{\xi}^2 -\frac{d}{6} e^{-4\kappa \phi} \absolute{H}^2\big] - \frac{1}{8}\:g\: e^{-4\kappa \phi}\absolute{H}^2\\
            &= \frac{1}{d-2}\left[\xi \otimes \xi - \frac{\absolute{\xi}^2}{2} g \right]  +e^{-4\kappa \phi} \frac{H^2}{4} +\left(\frac{d-2}{12(d-2)}- \frac{1}{8}\right)e^{-4\kappa \phi} \absolute{H}^2 \: g\\
            &= \frac{1}{d-2}\left[\xi \otimes \xi - \frac{\absolute{\xi}^2}{2} g \right]   + e^{-4\kappa \phi}\frac{H^2}{4}  -\frac{1}{24}e^{-4\kappa \phi} \absolute{H}^2 \: g.\\
        \end{split}
    \end{equation*}
    Furthermore,
    \begin{equation*}
        \begin{split}
            &\nabla^\mu T_{\mu\nu} \\ 
            &= \frac{1}{2(d-2)}\nabla^\mu\left[ 2\xi_\mu\xi_\nu- \xi^\lambda\xi_\lambda g_{\mu\nu}\right]  \\
            &\quad + \frac{1}{4}\nabla^\mu (e^{-4\kappa \phi} H_{\mu}^{\ \lambda\kappa} H_{\nu\lambda \kappa}) - \frac{1}{24} \nabla_\nu(e^{-4\kappa \phi}\absolute{H}^2) \\
            &= \frac{1}{2(d-2)}\Big[2(\nabla^\mu\xi_\mu)\xi_\nu \underbrace{+ 2 \xi_\mu(\nabla^\mu\xi_\nu)- 2(\nabla_\nu\xi^\lambda)\xi_\lambda}_{=0}\Big]  \\
            &\quad+ \frac{e^{-4\kappa \phi}}{4} \Big\{ -4\kappa(\nabla^\mu\phi) H_{\mu}^{\ \lambda\kappa} H_{\nu\lambda \kappa} - \extd^*H^{\lambda\kappa}H_{\nu\lambda\kappa} +H_\mu^{\ \lambda \kappa}\: \nabla^\mu H_{\nu\lambda\kappa} \\
            &\quad \frac{2\kappa}{3} (\nabla_\nu\phi)\absolute{H}^2 - \frac{1}{6} \nabla_\nu \absolute{H}^2 \Big\}\\
            &\stackrel{(D)}{=} \frac{e^{-4\kappa \phi}}{6(d-2)} \left\{-\absolute{H}^2 \: \xi_\nu -6 \xi^\mu H_{\mu}^{\ \lambda\kappa} H_{\nu\lambda \kappa} + \xi_\nu\absolute{H}^2\right\} \\
            &\quad+ \frac{e^{-4\kappa \phi}}{4} \left\{- \extd^*H^{\lambda\kappa}H_{\nu\lambda\kappa}+H_\mu^{\ \lambda \kappa}\: \nabla^\mu H_{\nu\lambda\kappa}  - \frac{1}{6} \nabla_\nu \absolute{H}^2 \right\}\\
            &\stackrel{(H)}{=} -\frac{e^{-4\kappa \phi}}{d-2} \xi^\mu H_{\mu}^{\ \lambda\kappa} H_{\nu\lambda \kappa} \\
            &\quad+ \frac{e^{-4\kappa \phi}}{4} \left\{\frac{4}{d-2} \xi^\mu H_\mu^{\ \lambda\kappa}H_{\nu\lambda\kappa} +H^{\mu \lambda \kappa} \nabla_\mu H_{\nu\lambda\kappa}  - \frac{1}{3} H^{\mu\lambda\kappa}(\nabla_\nu H_{\mu\lambda\kappa}) \right\}\\
            &=\frac{e^{-4\kappa \phi}}{4} \left\{ H^{\mu \lambda \kappa}\: \nabla_\mu H_{\nu\lambda\kappa} - \frac{1}{3} H^{\mu \lambda \kappa}\extd H_{\nu\mu\lambda\kappa} -  H^{\mu \lambda \kappa} \nabla_{[\mu} H_{\lambda\kappa]\nu} \right\}\\
            &= 0.\\
        \end{split}
    \end{equation*}
    The equality at the underbrace employs closedness of $\xi$. By (H) and (D) we denote that we made use of the $H$-field and dilaton equation respectively. The last equality employs closedness of $H$.
\end{proof}
Finally, for completeness, we determine the Einstein frame constraint equations.
\begin{proposition}
    The Einstein frame constraint equations are given by 
    \begin{equation*}
        \begin{split}
             \rscal_\Sigma + (\tr k)^2- \absolute{k}^2 &= \frac{1}{d-2}\left[\absolute{\extd^\Sigma \phi_0}^2+\phi_1^2 \right]  + \frac{e^{-4\kappa \phi_0}}{12}\left[\absolute{H_0}^2 + 3\absolute{h_0}^2 \right],\\
             \divergence_\Sigma k - \tr k &= \frac{1}{d-2} \phi_1 \extd\phi_0  + \frac{e^{-4\kappa \phi_0}}{4} C(h_0,H_0),\\
             0 &=  \left[(\extd^\Sigma)^* + \frac{4}{d-2}i_{\extd^\Sigma \phi_0}\right] h_0.
        \end{split}
    \end{equation*}
    Herein, $C(h_0,H_0)$ is the one-form obtained from contraction of $h_0$ and $H_0^\parallel$ as described below (\ref{genconstrainteqs}). 
\end{proposition}
\begin{proof}
    It is more efficient to obtain these equations from the Einstein frame GEE (\ref{einsteinframecombinedeinsteineqs}) directly than to apply the transformation rules for the initial data developed in Lemma \ref{initialdata_einsteinframe_lemma} to the string frame constraint equations (\ref{genconstrainteqs}). 

    The first constraint equation in (\ref{genconstrainteqs}) is equivalent\footnote{Note that the string frame equation $G[g] = T[g,B,\phi]$ and the Einstein frame equation $G[\tilde{g}] = \tilde{T}[\tilde{g},B,\phi]$ are component-wise equivalent, as $\tilde{T}$ is defined by the equation $G[\tilde{g}] - G[g] = \tilde{T}[\tilde{g},B,\phi] - T[g,B,\phi]$. Assuming the dilaton equation, we have computed an equivalent expression for $\tilde{T}$ in Proposition \ref{emtensorprop}.} to $G(N,N) = T(N,N)$, employing that $G(N,N) = \frac{1}{2}[\rscal_\Sigma + (\tr k)^2- \absolute{k}^2]$. Noting that
    \begin{equation*}
        \begin{split}
            T(N,N) &= \frac{1}{d-2}\left[\xi \otimes \xi - \frac{\absolute{\xi}^2}{2} g\right](N,N)  + \frac{e^{-4\kappa \phi}}{4}\big[H^2  -\frac{\absolute{H}^2}{6} \: g\big](N,N)\\
            &= \frac{1}{d-2}\left[(\phi_1)^2 + \frac{\absolute{\extd^\Sigma \phi_0}^2-\phi_1^2}{2} \right]  + \frac{e^{-4\kappa \phi_0}}{4}\left[\absolute{h_0}^2  +\frac{\absolute{H_0}^2 - 3\absolute{h_0}^2}{6} \right]\\
            &= \frac{1}{2(d-2)}\left[\absolute{\extd^\Sigma \phi_0}^2+\phi_1^2 \right]  + \frac{e^{-4\kappa \phi_0}}{24}\left[\absolute{H_0}^2 + 3\absolute{h_0}^2 \right],
        \end{split}
    \end{equation*}
    the first Einstein frame constraint equation follows. Similarly, we note that the second constraint in (\ref{genconstrainteqs}) is equivalent to $G(N,X) = T(N,X)$ for $X \in \Gamma(T\Sigma)$, employing that $G(N,X) = (\divergence_\Sigma k)(X) - X(\tr k)$. We note that
    \begin{equation*}
        \begin{split}
            T(N,X) &=  \frac{1}{d-2}\left[\xi \otimes \xi - \frac{\absolute{\xi}^2}{2} g\right](N,X)  + \frac{e^{-4\kappa \phi}}{4}\big[H^2  -\frac{\absolute{H}^2}{6} \: g\big](N,X)\\
            &=  \frac{1}{d-2} \phi_1 \extd\phi_0(X)  + \frac{e^{-4\kappa \phi_0}}{4} C(h_0,H_0)(X)\\
        \end{split}
    \end{equation*}
    and obtain the second constraint equation. Finally, the third string-frame constraint in (\ref{genconstrainteqs}) is equivalent to $i_N \extd^* H = - \frac{4}{d-2}i_Ni_\xi H$ (compare this to the first equation in Proposition \ref{einsteinframeprop}). With Lemma \ref{extdhypersurflemma}, we obtain the constraint equation
    \begin{equation*}
        -(\extd^\Sigma)^*h_0 = \frac{4}{d-2}i_{\extd^\Sigma \phi_0} h_0.
    \end{equation*}
\end{proof}

\subsection{Gauge Conditions and Hyperbolic Reduction}\label{gaugeconditions_section}
Throughout this section, let $(M,g,H,\phi), \xi = \extd \phi,$ be a SuGra spacetime solving the Einstein frame GEE (\ref{einsteinframecombinedeinsteineqs}). Note that, for ease of notation, we drop the tilde in this section. We write $H = \Bar{H} + \extd B$ for some closed $\bar{H} \in \Omega^3_{\mathrm{cl}}(M)$ and a two-form $B \in \Omega^2(M)$. Again, we emphasize that in general, the Einstein frame can only be adopted locally after choosing a potential for the dilaton, $\xi = \extd \phi$, and one should therefore imagine the SuGra spacetime we consider here to be a small subset of some development.

Crucial for the well-posedness of the Cauchy problem is the well-behavedness of the equations of motion (\ref{einsteinframecombinedeinsteineqs}) for $g$, $B$ and $\phi$. As explained before, we want to understand them as being of the form (\ref{hyperbolicpdeeq}). In this section, we develop gauge conditions that achieve this.

The GEE are invariant under generalised diffeomorphisms (a combination of diffeomorphisms and $B$-field transformations). Implementing a gauge condition means breaking this invariance. To that end, we introduce a background metric\footnote{It is more precise to say that we introduce a background connection $\Bar{\nabla}$, as the metric itself will not appear in the gauge condition. However, we can assume this connection to be the Levi-Civita connection of a metric.}. (Conceptually, one should think of fixing the background field $\bar{H}$ only after breaking diffeomorphism invariance with the DeTurck gauge.) For now, it suffices to take any Lorentzian metric $\Bar{g}$ on $M = \reals \times \Sigma$. Define 
\begin{equation*}
    A^\mu_{\alpha\beta} \coloneqq (\nabla - \Bar{\nabla})^\mu_{\alpha\beta} \in \Omega^1(\End \:TM)
\end{equation*}
where $\nabla$ and $\Bar{\nabla}$ denote the Levi Civita connection of $g$ and $\Bar{g}$ respectively. Notice that, in local coordinates
\begin{equation*}
    A^\mu_{\alpha\beta} = \Gamma^\mu_{\alpha\beta} - \Bar{\Gamma}^\mu_{\alpha\beta}
\end{equation*}
Denoting
\begin{equation*}
    \Gamma_\mu \coloneqq g_{\mu\nu}g^{\alpha\beta} \Gamma^\nu_{\alpha\beta}, \qquad\quad F_\mu \coloneqq g_{\mu\nu} g^{\alpha\beta} \Bar{\Gamma}^\nu_{\alpha\beta}
\end{equation*}
one can check that
\begin{equation*}
    \ric_{\mu\nu} = -\frac{1}{2} g^{\alpha\beta} \partial_\alpha\partial_\beta g_{\mu\nu} + \partial_{(\mu} \Gamma_{\nu)} + \text{lower order}
\end{equation*}
It is now apparent how to modify the Ricci tensor. Following \cite[Chapter 14.1]{cauchy}, we define $\mathcal{D}_\nu \in \Gamma(T^*M)$ as
\begin{equation}\label{ddefeq}
    \mathcal{D}_\nu \coloneqq -g_{\nu\mu} g^{\alpha\beta} A^\mu_{\alpha\beta} = F_\nu - \Gamma_\nu 
\end{equation}
and then
\begin{equation*}
    \richat_{\mu\nu} = \ric_{\mu\nu} + \nabla_{(\mu} \mathcal{D}_{\nu)} = \ric_{\mu\nu} + \frac{1}{2} L_\mathcal{D} g_{\mu\nu},
\end{equation*}
where $L_{\mathcal{D}}$ denotes the Lie derivative in direction of the vector field $g^{\mu\nu}\mathcal{D}_\nu$. Citing \cite[eq. (14.3)]{cauchy}, we see that
\begin{equation}\label{richateq}
    \richat_{\mu\nu} = - \frac{1}{2} g^{\alpha\beta} \partial_\alpha \partial_\beta g_{\mu\nu} + \nabla_{(\mu} F_{\nu)} + g^{\alpha\beta} g^{\gamma\delta}[\Gamma_{\alpha\gamma\mu}\Gamma_{\beta\delta\nu} + \Gamma_{\alpha\gamma\mu}\Gamma_{\beta\nu\delta} + \Gamma_{\alpha\gamma\nu}\Gamma_{\beta\mu\delta}]
\end{equation}
where $\Gamma_{\alpha\beta\gamma} \coloneqq g_{\beta\mu} \Gamma^\mu_{\alpha\gamma}$, i.e.\ $\richat$ is indeed a hyperbolic differential operator of the form $g^{\mu\nu}\partial_\mu\partial_\nu + $lower order. 
\begin{remark}\label{deturck_gauge_rem}
    The gauge condition that we impose in this formalism is $\mathcal{D} = 0$, in which case $\richat = \ric$.  To highlight the dependence of $\mathcal{D}$ on the dynamical metric $g$ and the background metric $\bar{g}$, we denote $\mathcal{D} = \mathcal{D}[g,\bar{g}]$. Then, of course we have the diffeomorphism invariance $f^*(\mathcal{D}[g,\bar{g}]) = \mathcal{D}[f^*g, f^*\bar{g}]$. However, viewing the background metric as fixed, we also denote $\mathcal{D}[g] = \mathcal{D}[g,\bar{g}]$. Then, the condition $\mathcal{D} = 0$ breaks diffeomorphism invariance, because in general 
    \[f^*(\mathcal{D}[g]) = \mathcal{D}[f^*g,f^*\bar{g}] \neq \mathcal{D}[f^*g,\bar{g}] = \mathcal{D}[f^*g]. \]
\end{remark}

Inspired by how we modified the Ricci-tensor, we take a look at the operator $\extd^*\extd$. Acting on $B$, and summarising all terms which do not contain second derivatives of $g_{\mu\nu}$ or $B_{\mu\nu}$ under the term \enquote{lower order}, we get
\begin{equation*}
    \begin{split}
        \extd^* \extd B_{\mu\nu} &= -3 \nabla^\lambda \nabla_{[\lambda} B_{\mu\nu]} \\
        &=- 3 \nabla^\lambda \partial_{[\lambda} B_{\mu\nu]} \\
        &=-3 \partial^\lambda \partial_{[\lambda} B_{\mu\nu]} + \text{lower order} \\
        &=- \partial^\lambda \partial_\lambda B_{\mu\nu}- 2 \partial^\lambda \partial_{[\mu} B_{\nu]\lambda} + \text{lower order}\\
        &= - \partial^\lambda\partial_\lambda B_{\mu\nu} - 2 \partial_{[\mu} (\partial^\lambda B_{\nu]\lambda}) + \text{lower order}.\\
    \end{split}
\end{equation*}
We denoted $\partial^\mu = g^{\mu\nu} \partial_\nu$. 

The second term contains precisely those second derivatives of the $B$-field that are problematic, as they cause the operator to not be of the form (\ref{hyperbolicpdeeq}). Notice that these terms equal 
\begin{equation*}
    \extd\extd^*B = 2 \nabla_{[\mu} \nabla^\lambda B_{\nu]\lambda}    
\end{equation*}
up to second order derivatives of the metric. These, however, we may not add to our equation, as this would only result in the equation differing from the desired form (\ref{hyperbolicpdeeq}) in a different way. Luckily, we have another metric at our disposal: $\Bar{g}$. And we don't care about its second derivatives entering into our equation. Thus, we identify as the quantity by which we want to modify the $B$-field equation
\begin{equation}\label{cdefeq}
    \begin{split}
        \mathcal{C}_\nu &\coloneqq \extd^*_{g,\Bar{g}} B_\nu \coloneqq -g^{\kappa\lambda}\Bar{\nabla}_\kappa B_{\lambda\nu} \in \Gamma(T^*M)\\
        & = \extd^* B_\nu + \mathcal{D}^\mu B_{\mu\nu} + g^{\kappa\lambda} A^\mu_{\kappa \nu} B_{\lambda \mu}.
    \end{split}
\end{equation}
To highlight the dependence of $\mathcal{C}$ on $g$, $B$, and $\bar{g}$, we sometimes write $\mathcal{C} = \mathcal{C}[g,B]$ or even $\mathcal{C} = \mathcal{C}[g,B,\bar{g}]$. 

Finally, by replacing in (\ref{einsteinframecombinedeinsteineqs})
\begin{equation*}
    \begin{split}
         \ric &\quad \longrightarrow \quad  \richat,  \\
         \extd^*\extd B &\quad \longrightarrow \quad \extd^*\extd B + \extd\mathcal{C} \eqqcolon -\hat{\Box}_{\mathrm{Hd}} B,
    \end{split}
\end{equation*}
we obtain the modified system
\begin{equation}\label{modifiedcombinedeinsteineqs}
    \begin{split}
        \hat{\Box}_{\mathrm{Hd}}B &= \extd^*\Bar{H} +\frac{4}{d-2} i_\xi H \\
        \richat &= \frac{1}{d-2} \left[\xi \otimes \xi - \frac{e^{-4\kappa \phi}}{6}\absolute{H}_{g}^2\: g \right]+ e^{-4\kappa \phi}\frac{H^{2,g}}{4}  \\ 
        \Box\phi &= -\frac{e^{-4\kappa \phi}}{6} \absolute{H}^2_{g}  
    \end{split}
\end{equation}
where $H = \Bar{H} + \extd B$, $\xi = \extd \phi$. 
\begin{proposition}\label{localformprop}
    The modified Einstein frame GEE (\ref{modifiedcombinedeinsteineqs}) form a quasi-linear hyperbolic system with principal symbol given by $g^{\mu\nu}$. Furthermore, writing in local coordinates the system as $g^{\mu\nu}\partial_\mu \partial_\nu - f[u] = 0$ where $u = (g_{\mu\nu},B_{\mu\nu},\phi)$, we see the non-linearity $f[u](t,x)$ to vanish if the following expressions are all set to vanish at $(t,x)$:
    \begin{itemize}
        \item $\partial_\alpha g_{\mu\nu}$, $B_{\mu\nu}$, $\partial_\alpha B_{\mu\nu}$, $\phi$, and $\partial_\alpha \phi$,
        \item $\bar{H}_{\mu\nu\lambda}$, $ \partial_\alpha\bar{H}_{\mu\nu\lambda}$, $\partial_\alpha\bar{g}_{\mu\nu}$, and $\partial_\alpha\partial_\beta\bar{g}_{\mu\nu}$. 
    \end{itemize}
    In particular, the modified GEE (\ref{modifiedcombinedeinsteineqs}) are of the form (\ref{localexistenceproof_hyperbolicpdeeq}) required for the local existence and uniqueness result Proposition \ref{localexistenceprop_preliminaries} to apply.
\end{proposition}
\begin{proof}
    In the following, we refer to the demands on the principal symbol and the non-linearity as the equation \enquote{being of the right form}. For the Ricci-equation, the right form can be seen from the expression (\ref{richateq}) for the modified Ricci tensor $\richat$ in local coordinates. For the dilaton equation, the right form follows from the local expression of the d'Alembert operator $\Box \phi = g^{\mu\nu}\partial_\mu\partial_\nu \phi -  g^{\lambda\kappa}\Gamma^\alpha_{\lambda\kappa}\partial_\alpha\phi$. Finally, for the $B$-field equation, we compute
    \begin{equation*}
        \begin{split}
            \hat{\Box}_{\mathrm{Hd}} B_{\mu\nu} &= \partial^\lambda \partial_\lambda B_{\mu\nu} - \Gamma^\alpha \partial_\alpha B_{\mu\nu} + 2 g^{\lambda\kappa} \partial_{[\mu}(\bar{\Gamma}^\alpha_{\nu]\lambda} B_{\alpha\kappa}+ B_{\nu]\alpha}\bar{\Gamma}^\alpha_{\lambda \kappa}) - 2 g^{\lambda\kappa}B_{\beta\kappa} \Gamma^\alpha_{\lambda[\mu} \bar{\Gamma}^\beta_{\nu]\alpha}\\
            &\quad- 4g^{\lambda\kappa} B_{\beta[\mu} \Gamma^\alpha_{\nu]\lambda} \bar{\Gamma}^\beta_{\alpha\kappa} + 2 g^{\lambda\kappa} \Gamma^\beta_{\kappa[\mu}\partial_{|\lambda|}B_{\nu]\beta} - 2 g^{\lambda\kappa} B_{\alpha\beta} \Gamma^\beta_{\kappa[\mu}\bar{\Gamma}^\alpha_{\nu]\lambda},
        \end{split}
    \end{equation*}
    where we denoted $\partial^\lambda \coloneqq g^{\lambda\kappa}\partial_\kappa$. The right form follows.
\end{proof}
In the end, the modified system has to agree with the unmodified system, i.e.\ $\extd\mathcal{C} = 0$, since our actual interest lies in finding solutions to the latter. (For the same reason, we will also have $\mathcal{D} = 0$.) Thus, the gauge that we implement can be understood as 
\begin{equation}\label{gen_Lor_gauge:eq}
    0 = \extd \mathcal{C} =\extd \extd^*_{g,\Bar{g}} B
\end{equation}
We call this condition the \emph{generalised Lorenz gauge}. Note that it breaks invariance of the equations under closed $B$-field transformations, as $\extd\mathcal{C}[g,B] \neq \extd\mathcal{C}[g,B+\beta]$ for $\beta$ a closed two-form. We prove in Lemma \ref{bfieldgaugelemma} that one can always implement $\extd\mathcal{C}=0$ by virtue of a closed $B$-field transformation.
\begin{remark}\label{lorenzgauge_comparison_rem}
    We can compare the generalised Lorenz gauge (\ref{gen_Lor_gauge:eq}) to the Lorenz gauge $\extd^* A = 0$ as it is imposed on the electromagnetic field potential\footnote{For the purposes of this remark, we forget that the letter $A$ is already occupied by the object $A = \nabla -\Bar{\nabla}$.} $A$ in the study of the Cauchy Problem for the Einstein-Maxwell system in \cite{bruhatGRBook}. For a one-form, one can calculate that $\extd^* A = -\partial^\mu A_\mu + \Gamma^\mu A_\mu$. That is, the difference $\extd^*_{g,\Bar{g}}A - \extd^* A$ is given by $\mathcal{D}^\mu A_\mu$. As one also has to implement the gauge $\mathcal{D} = 0$ in the study of the Cauchy problem, one can equivalently state the Lorenz gauge as $\extd^*_{g,\Bar{g}} A = 0$. Noting that for a $0$-form being closed implies being constant, we see that, for a one form, the generalised Lorenz gauge $\extd \extd^*_{g,\bar{g}} A = 0$ is a mild generalisation of the Lorenz gauge.
\end{remark}
We now establish that one can obtain solutions to the original system (\ref{einsteinframecombinedeinsteineqs}) by studying the modified one (\ref{modifiedcombinedeinsteineqs}). Recall that under the assumptions of Lemma \ref{metricgaugefield_sourcefreewaveeq_lemma_preliminaries}, $\mathcal{D}$ vanishes on the entire Cauchy development of any subset $\Omega$ of the initial hypersurface $\Sigma$ on which $\mathcal{D}= 0$ and $\nabla\mathcal{D}= 0$. However, this result does not immediately apply to a solution of the modified GEE (\ref{modifiedcombinedeinsteineqs}). While we know from Proposition \ref{emtensorprop} that the energy-momentum tensor $T$ obtained in the Einstein frame is divergence-free if one assumes the $H$-field and dilaton equation, a priori this does not hold after one passes to the modified system (\ref{modifiedcombinedeinsteineqs}), because then one also modifies the $H$-field equation. The solution to this is to first prove the vanishing of the modifying term $\extd \mathcal{C}$ of the $H$-field equation under suitable conditions.
\begin{proposition}\label{bfieldgaugefield_sourcefreewaveeq_prop}
    Let $(M,g, H, \phi)$, $H= \Bar{H} + \extd B$, $\xi = \extd \phi$, be a globally hyperbolic SuGra spacetime such that (\ref{modifiedcombinedeinsteineqs}) is satisfied. Let $\Bar{g}$ be an arbitrary Lorentz metric on $M$, and define $\mathcal{D}$ and $\mathcal{C}$ as in (\ref{ddefeq}) and (\ref{cdefeq}), respectively.
    
    Then, assuming $\mathcal{D} = 0$, and $\nabla\mathcal{D} = \extd\mathcal{C} = 0$ on some subset $\Omega \subset \Sigma$, we have $\mathcal{D} = 0$ and $\extd\mathcal{C} = 0$ on the entire Cauchy development $D(\Omega)$.
\end{proposition}
\begin{proof}
    We begin by deriving a source-free wave equation for $\mathcal{C}$. We take the co-differential of the modified $B$-field equation:
    \begin{equation}\label{cwaveeq}
        \begin{split}
            0 &= \extd^*\left[\extd^* H + \extd \mathcal{C} + \frac{4}{d-2} i_\xi H\right] \\
            &= \extd^*\extd \mathcal{C} - \frac{4}{d-2} i_\xi \extd^*H \\
            &= \extd^*\extd\mathcal{C} +\frac{4}{d-2} i_\xi\left(\extd \mathcal{C} + \frac{4}{d-2} i_\xi H\right) \\
            &= \extd^*\extd\mathcal{C} +\frac{4}{d-2} i_\xi \extd \mathcal{C}
        \end{split}
    \end{equation}
    Herein, we have used the antisymmetry of $H$ to conclude $i_\xi i_\xi H = 0$ and
    \begin{equation*}
        \extd^*(i_\xi H)_\nu = -\nabla^\mu [(\nabla^\lambda\phi) H_{\lambda \mu \nu}] = (\nabla^\lambda \phi) \nabla^\mu H_{\mu\lambda\nu} = -(i_{\xi} \extd^*H)_\nu.
    \end{equation*}
    We note that (\ref{cwaveeq}) has indeed no source terms. Thus, by Theorem \ref{coexact_wave_eqs_main_thm}, we must have $\extd\mathcal{C} = 0$ on $D(\Omega) = M$. 

    Finally, the tensor $T = \richat - \frac{1}{2}\trwith{g}\richat$ agrees with the energy momentum tensor obtained in Proposition \ref{emtensorprop} in the Einstein frame. It follows from the (now effectively unmodified) $H$-field and dilaton equations that $\divergence T = 0$. Hence, by Lemma \ref{metricgaugefield_sourcefreewaveeq_lemma_preliminaries}, $D = 0$ on $D(\Omega)$. 
\end{proof}

\subsection{Setting Up Initial Values}\label{intialdatasection}
In this section, we set up initial values for the modified system (\ref{modifiedcombinedeinsteineqs}). This means that, given initial data $\mathcal{I} = (\Sigma,g_0,k,H_0,h_0,\phi_0,\phi_1)$ to the Einstein frame GEE (\ref{einsteinframecombinedeinsteineqs}), we construct a manifold $M$ on which the development may take place, equipped with a background metric $\Bar{g}_{\mathcal{I}}$, a background field $\Bar{H}_{\mathcal{I}} \in \Omega^3_{\mathrm{cl}}(M)$, and on the initial hypersurface $\Sigma$ the metric $g$, the $B$-field, the dilaton $\phi$, and their respective first derivatives, all such that with $H = \bar{H}_{\mathcal{I}} + \extd B$ the initial data induced on $\Sigma$ by $(M,g,H,\phi)$ reproduces the given initial data. We again drop the tildes over the objects in the Einstein frame. (We remark on this as we reinstate them in subsequent sections.)  

We choose to highlight the dependence of the background fields on the initial data in our notation. This is motivated by the Einstein frame initial data only existing locally and with respect to a choice of potential for the dilaton, in contrast to the situation present in the Einstein equations coupled to a scalar field as discussed in \cite{cauchy} where well-defined global initial data and background fields exist.

Note we don't have initial data for the $B$-field directly, but only for its exterior derivative $\extd B$. However, in the discussion of uniqueness of solutions, it is necessary to have better control over $B$, and so we artificially introduce two-forms $B_0, B_1 \in \Omega^2(\Sigma)$ that are compatible with the requirements on $\extd B$ and demand $B^\parallel = B_0$ and $(\nabla_N B)^\parallel - \extd^\Sigma(B(N)) = B_1$, where $N$ is the unit normal on $\Sigma$. With these additional requirements, we can determine most degrees of freedom in the initial values of $g$, $B$, and $\phi$.

However, in the specification of the normal parts of $g$ and $B$ and their normal derivatives, there still appear degrees of freedom that are not determined by the initial data. For the metric, this is observed e.g.\ in \cite[Chapter 14.2]{cauchy}. The reason is that the induced initial data is invariant under the action of diffeomorphisms on $M$ which restrict to the identity on $\Sigma$. The freedom in specifying the normal parts of $g$ and its normal derivative can be used to specify the unit normal vector and force $\mathcal{D} = 0$ initially. For the $B$-field, the initial data induced is invariant even under the action of closed $B$-field transformations which restrict to the identity on the generalised tangent bundle over $\Sigma$. That is, we can add to $B$ any closed two-form $\beta \in \Omega^2(M)$ such that $\beta^\parallel = 0$. We use this freedom to achieve $B(N) = 0$ and $\mathcal{C}^\parallel = 0$ on $\Sigma$.

Crucially, the way we set up the initial data is geometrical and does not depend on a choice of coordinates. Following \cite{cauchy}, this allows the glueing together of developments defined only on a small neighbourhood of points in $\Sigma$, as it forces these developments to agree on the initial hypersurface (where both are defined). Taking this one step further, we have to set up the initial data independent of the choice of Einstein frame, so that we can glue together developments defined in different Einstein frames. Note that two dilaton potentials $\phi$ and $\phi'$ to a given dilaton $\xi$ defined over the same connected set differ by a constant. 

Let us define the background manifold and the background fields. Let $M = \reals \times \Sigma$, and identify $\Sigma \cong \{0\} \times \Sigma$. Denote by $t \colon \reals \to \Sigma$ the coordinate coming from projection onto the first factor. We define the background fields as
\begin{equation}\label{backgroundfieldsdefeq}
    \Bar{g}_{\mathcal{I}} \coloneqq -e^{-2\kappa\phi_0}\extd t^2 + g_0, \qquad\quad  \Bar{H}_{\mathcal{I}} \coloneqq H_0 + \extd t \wedge h_0 + t\: \extd^\Sigma h_0
\end{equation}
where one can check $\bar{H}_{\mathcal{I}}$ to indeed be closed. 

In \cite[Chapter 14.2]{cauchy}, the initial values of $g$ and $\phi$ were fixed by imposing the requirements stated in the following Lemma - except for the factor $e^{\kappa \phi_0}$ in the first condition. This factor is needed to achieve invariance of the conditions under a change of Einstein frame. 
\begin{lemma}\label{initialdata_setup_metric_phi_lemma}
    The initial values for $g$, $\phi$, and their first time-derivatives are uniquely determined (and not overdetermined) by the following requirements.
    \begin{enumerate} [label = (\arabic*)]
        \item $e^{\kappa\phi_0}\partial_t$ agrees with the unit normal $N$ on $\Sigma$.
        \item $g$ induces the initial data $(g_0,k)$ on $\Sigma$, i.e.\ $g^\parallel = g_0$ and $2 k = L_N g$.
        \item $\phi$ induces the initial data $(\phi_0,\phi_1)$ on $\Sigma$, i.e.\ $\restr{\phi}{\Sigma} = \phi_0$ and $N(\phi) = \phi_1$.
        \item Initially, $g$ satisfies DeTurck's gauge, i.e.\ $\restr{\mathcal{D}[g,\bar{g}_{\mathcal{I}}]}{\Sigma} = 0$.
    \end{enumerate}
\end{lemma}
\begin{proof}
    To discuss this more explicitly, pick a coordinate chart $(U_\Sigma,x^m)$ in $\Sigma$ and extend it to coordinates $(x^0 = t, x^m)$ on $\reals \times U_\Sigma$. We specify the metric $g$ and its first derivatives on $U_\Sigma$ in these coordinates. The spatial part of $g$ is fixed by (2):
    \begin{equation}\label{initialdataconstrgdef1}
        \restr{g_{mn}}{t=0} = (g_0)_{mn} 
    \end{equation}
    Requirement (1) is equivalent to 
    \begin{equation}\label{initialdataconstrgdef2}
        \restr{g_{00}}{t=0} = -e^{-2\kappa\phi_0}, \qquad\quad \restr{g_{0n}}{t=0} = 0.
    \end{equation}
    From the formula for the second fundamental form $K_{mn} =\frac{1}{2} \absolute{g_{00}}^{-1/2}g_{mn,0}$, it is immediate that (2) is equivalent to
    \begin{equation}\label{initialdataconstrgdef3}
        \restr{g_{mn,0}}{t=0} = 2e^{-\kappa \phi_0} k_{mn}
    \end{equation}
    It remains to specify $g_{00,0}$ and $g_{0k,0}$, whose values we will see to be determined by (4). We make the following auxiliary computations for $t=0$:
    \begin{equation*}
        \begin{split}
            \Gamma^0_{00} &= \frac{1}{2}g^{00}g_{00,0} = -\frac{1}{2}e^{2\kappa\phi_0} g_{00,0},\\
            \Gamma^k_{00} &= -\frac{1}{2}g^{kl}(g_{00,l}-2g_{0l,0}) = g^{kl}(-\kappa e^{-2\kappa\phi_0}\partial_l\phi + g_{0l,0}), \\
            \Gamma^0_{mn} &= \frac{1}{g_{00}}g(\nabla_m \partial_n, \partial_0) =- e^{\kappa\phi_0} g(\nabla_m \partial_n, N) = e^{\kappa\phi_0}k_{mn}, \\
            \Gamma^k_{mn} &= -\frac{1}{2}g^{kl}(g_{mn,l}-g_{lm,n}-g_{nl,m}) = (\Gamma^\Sigma)^k_{mn}.
        \end{split}
    \end{equation*}
    Herein, $(\Gamma^\Sigma)^k_{mn}$ denote the Christoffel symbols for $g_0$. From this, it follows that
    \begin{equation*}
        \begin{split}
            \Gamma_0 &= g_{00} g^{\alpha\beta} \Gamma^0_{\alpha\beta} = \Gamma^0_{00} - e^{-\kappa\phi_0} \tr k = -\frac{1}{2}e^{2\kappa\phi_0} g_{00,0} - e^{-\kappa\phi_0} \tr k,\\
            \Gamma_l &= g_{lk} g^{\alpha\beta} \Gamma^k_{\alpha\beta} = \Gamma^\Sigma_l + g_{lk} g^{00} \Gamma^k_{00} = \Gamma^\Sigma_l + \kappa \partial_l\phi_0 - e^{2\kappa\phi_0} g_{0l,0} .
        \end{split}
    \end{equation*}
    Herein, $\Gamma^\Sigma_l$ denotes the contracted Christoffel symbol for $g_0$. We see that initially $0 = \mathcal{D}_\mu = F_\mu - \Gamma_\mu$ can be easily achieved by requiring
    \begin{equation*}
        \begin{split}
            \restr{g_{00,0}}{t=0} &= \restr{-2e^{-2\kappa\phi_0}[F_0 +e^{-\kappa\phi_0}\tr k ]}{t=0} ,\\
            \restr{g_{0k,0}}{t=0} &= \restr{e^{-2\kappa\phi_0}\left[\Gamma^\Sigma_k - F_k + \kappa \partial_l \phi_0 \right]}{t=0}.
        \end{split}
    \end{equation*}
    We finish the proof by noting that, trivially, (3) uniquely determines the initial value of $\phi$ and its first time-derivative. 
\end{proof}
We denote by $\nabla$ the Levi-Civita connection of $g$, and by $N = e^{\kappa \phi_0}\partial_t$ the unit normal on $\Sigma$. We denote by $B_0,B_1 \in \Omega^2(\Sigma)$ initial data for the $B$-field. Note that $\bar{H}_{\mathcal{I}}$ induces the initial data $(H_0,h_0)$, so that the initial requirement for $B$ is $\extd B = 0$. Thus, we have the canonical choice $B_0 = B_1 = 0$ available.
\begin{lemma}\label{initialdata_setup_bfield_lemma}
    The initial values for $B$ and its first derivative are uniquely determined (and not overdetermined) by the following requirements.
    \begin{enumerate}[label=(\arabic*)]
        \item $B$ induces the initial data $(B_0,B_1)$, i.e.\ $B^\parallel = B_0$ and $(\nabla_N B)^\parallel = B_1$.
        \item Initially, $B$ has no normal component, i.e.\ $B(N) = 0$.
        \item Initially, $B$ satisfies $(\mathcal{C}[g,B,\bar{g}_{\mathcal{I}}])^\parallel = 0$.
    \end{enumerate}    
    Finally, in the case $B_0 = B_1 = 0$, these requirements are solved by demanding $B$ and its time derivative to vanish initially.
\end{lemma}
\begin{proof}
    Write, implementing (1),
    \begin{equation*}
        \restr{B}{t=0} = B_0 - N^\flat \wedge b_0, \qquad\quad \restr{\nabla_N B}{t=0} = B_1 - N^\flat \wedge b_1. 
    \end{equation*}
    Requirement (2) yields $b_0 = 0$. It remains to determine $b_1$. Consider that by definition of $\mathcal{C}$ (\ref{cdefeq}) and Lemma \ref{extdhypersurflemma}
    \begin{equation*}
        \begin{split}
            \mathcal{C}^\parallel &= (\extd^*_{g,\Bar{g}_{\mathcal{I}}}B)^\parallel = (\extd^*B)^\parallel + E[B] \\
            &= (\extd^\Sigma)^*B^\parallel_0 + b_1 + E[B]  \\
        \end{split}
    \end{equation*}
    where $E[B]$ is some covector valued expression homogeneous and linear in $B$. To achieve that $\mathcal{C}^\parallel$ vanishes initially, we thus set
    \begin{equation*}
        b_1 =  - (\extd^\Sigma)^*B_0^\parallel - E[B].
    \end{equation*}
    Finally, considering the special case $B_0 = B_1 = 0$, we arrive at $\restr{B}{t=0} = \restr{\nabla_N B}{t=0} = 0$.
\end{proof}
The requirements of Lemmas \ref{initialdata_setup_metric_phi_lemma} and \ref{initialdata_setup_bfield_lemma} leave no degrees of freedom left to specify. However, due to the constraint equations, we get the following
\begin{lemma}\label{gaugefieldinitialvanishinglemma}
    Let $\mathcal{I}= (\Sigma,g_0,k,H_0, h_0,\phi_0,\phi_1)$ be initial data for the Einstein frame system (\ref{einsteinframecombinedeinsteineqs}). Assume that we can find a development $(M,g,H,\phi)$ of the initial data under the modified system (\ref{modifiedcombinedeinsteineqs}) satisfying the assumption of Lemmas \ref{initialdata_setup_metric_phi_lemma} and \ref{initialdata_setup_bfield_lemma}. Then we have initially
    \begin{equation*}
        \restr{\nabla\mathcal{D}}{\Sigma} = 0, \qquad\quad \restr{\extd\mathcal{C}}{\Sigma} = 0.
    \end{equation*}
\end{lemma}
\begin{proof}
    Since we assume the initial data to solve the constraint equations, we have on $\Sigma$
    \begin{equation*}
        G_{0m} = T_{0m}, \qquad G_{mn} = T_{mn}, \qquad \extd^*\extd B_{0n} = - \extd^*(\bar{H}_{\mathcal{I}})_{0n} -\frac{4}{d-2} \xi^\lambda H_{\lambda0n}
    \end{equation*}
    Denote $\hat{G}_{\mu\nu} = \richat_{\mu\nu} - \frac{1}{2}(\trwith{g}\richat)g_{\mu\nu}$. Since we assume $(M,g,H,\phi)$ to be a solution of the modified system (\ref{modifiedcombinedeinsteineqs}), we can conclude that on $\Sigma$
    \begin{equation*}
        0 = \hat{G}_{0m} - T_{0m} = \hat{G}_{0m} - G_{0m} = \nabla_{(0}\mathcal{D}_{m)} - \frac{1}{2} g_{0m} \nabla^\lambda \mathcal{D}_\lambda = \frac{1}{2}\partial_0 \mathcal{D}_m
    \end{equation*}
    where the last equality follows from $\restr{\partial_t}{\Sigma}$ being normal to $\Sigma$ and $\restr{\mathcal{D}}{\Sigma} = 0$ initially. Furthermore on $\Sigma$
    \begin{equation*}
        0 = \hat{G}_{mn} - T_{mn} = \hat{G}_{mn} - G_{mn} = \nabla_{(m} \mathcal{D}_{n)}- \frac{1}{2} g_{mn} \nabla^\lambda \mathcal{D}_\lambda =- \frac{1}{2} g_{mn}g^{00} \partial_0 \mathcal{D}_0
    \end{equation*}
    where, in the last equality, we employed again that $\mathcal{D} = 0$ initially. Therefore, $\restr{\nabla \mathcal{D}}{\Sigma} = 0$.

    By a similar argument, we see that on $\Sigma$
    \begin{equation*}
        0 = \hat{\Box}_{\mathrm{Hd}}B_{0n} + \extd^*\extd B_{0n} =- (i_{\partial_t}\extd \mathcal{C})_n
    \end{equation*}
    From the required $\mathcal{C}^\parallel = 0$ (cf.\ Lemma \ref{initialdata_setup_bfield_lemma}), we get that
    \begin{equation*}
        (\extd\mathcal{C})^\parallel = 0.
    \end{equation*}
    Put together, $\restr{\extd\mathcal{C}}{\Sigma} = 0$ as claimed.
\end{proof}
Let us conclude the discussion of the initial data by checking that our constructions are in fact invariant under a change of Einstein frame.
\begin{lemma}\label{dilaton_gaugetrafo_lemma}
    Let $\mathcal{I} = (g_0,k,H_0,h_0,\phi_0,\phi_1)$ be initial data in an Einstein frame, and denote by $\bar{g}_{\mathcal{I}}$ and $\bar{H}_{\mathcal{I}}$ the associated background fields. Let $(g,H,\phi)$, $H = \bar{H}_{\mathcal{I}}+\extd B$, be a SuGra spacetime. Consider with $c \in \reals$ a change of Einstein frame
    \begin{equation*}
        \begin{split}
            (g',B',\phi') &\coloneqq (e^{-2\kappa c}g,B,\phi+c) \\
            \mathcal{I}' = (g_0',k',H_0',h_0',\phi_0',\phi_1')  &\coloneqq (e^{-2\kappa c}g_0,e^{-\kappa c}k,H_0, e^{\kappa c}h_0,\phi_0+c,e^{\kappa c}\phi_1). \\
        \end{split}
    \end{equation*}
    Then, the following hold.
    \begin{enumerate}[label=(\roman*)]
        \item $\mathcal{D}[g',\bar{g}_{\mathcal{I}'}] = \mathcal{D}[g,\bar{g}_{\mathcal{I}}]$. In particular, DeTurck's gauge condition $\mathcal{D} = 0$ is invariant under a change of Einstein frame.
        \item $\mathcal{C}[g',B',\bar{g}_{\mathcal{I}'}] = e^{2\kappa c} \mathcal{C}[g,B,\bar{g}_{\mathcal{I}}]$.
        In particular, the generalised Lorenz gauge condition $\extd\mathcal{C} = 0$ is invariant under a change of Einstein frame.
        \item If $(g,B,\phi)$ satisfies the conditions from Lemma \ref{initialdata_setup_metric_phi_lemma} with respect to $\mathcal{I}$, then $(g',B',\phi')$ satisfies those conditions with respect to $\mathcal{I}'$.
        \item If $(g,B,\phi)$ satisfies the conditions from Lemma \ref{initialdata_setup_bfield_lemma} with respect to $\mathcal{I}$, then $(g',B',\phi')$ satisfies those conditions with respect to $\mathcal{I}'$.
    \end{enumerate}
\end{lemma}
\begin{proof}
    Note first that
    \begin{equation*}
        \bar{g}_{\mathcal{I}'} =e^{-2\kappa c}\bar{g}_{\mathcal{I}}, \qquad\quad \bar{H}_{\mathcal{I}'}= \bar{H}_{\mathcal{I}}.
    \end{equation*}
    \textbf{(i) \& (ii):} Recall that a conformal change with constant conformal factor leaves the associated LC connection invariant. The claims respectively follow from the definitions (\ref{ddefeq}) and (\ref{cdefeq}) of $\mathcal{D}$ and $\mathcal{C}$.
    
    \textbf{(iii):} We check conditions (1-4) for $(g',B',\phi')$. Condition (4) follows from assertion (i). Conditions (2) and (3) follow from Lemma \ref{initialdata_einsteinframe_lemma}. Finally, (1) demands that $N' =  e^{\kappa\phi'_0}\partial_t$, where $N'$ is the unit normal on $\Sigma$ with respect to $g'$. This holds because orthogonality is preserved by conformal transformations and by (1) for $(g,B,\phi)$
    \begin{equation*}
        g'(e^{\kappa\phi'_0}\partial_t,e^{\kappa\phi'_0}\partial_t) = g(e^{\kappa\phi_0}\partial_t,e^{\kappa\phi_0}\partial_t) = g(N,N) = -1,
    \end{equation*}
    where $N$ denotes the unit normal on $\Sigma$ with respect to $g$.
    
    \textbf{(iv):}  We check conditions (1-3) for $(g',B',\phi')$. Condition (1) follows from Lemma \ref{initialdata_einsteinframe_lemma}. Condition (3) follows from assertion (ii). Finally, (2) follows from $B' = B$ and $N' = e^{\kappa c}N$. 
\end{proof}

\subsection{Existence of a Globally Hyperbolic Development}\label{existence_ghd_section}

In this section, we prove the existence of a globally hyperbolic development to initial data for the GEE in the string frame (\ref{combinedgeneinsteineqs}). In the setting of the Einstein equations coupled to a scalar field, the corresponding result can for example be found in \cite[Theorem 14.2.]{cauchy}, and we closely follow the proof strategy established there. To minimise repetition and highlight our contributions, we present results established in the proof of \cite[Theorem 14.2.]{cauchy} in the preliminary Section \ref{preliminaries_section}. 

From Proposition \ref{localformprop}, we know the modified Einstein frame GEE to be locally of the form required for Proposition \ref{localexistenceprop_preliminaries} to apply. Proposition \ref{localexistenceprop_preliminaries} states that, in a neighbourhood of any point $p\in \Sigma$, there is a solution to the modified Einstein frame GEE. We apply it to obtain the following local existence result for the string frame GEE. Note that the result remembers the gauge-fixed Einstein frame development associated to the obtained string frame development. This is useful for patching together of solutions, as we only have good control over solutions of the modified Einstein frame GEE.
\begin{lemma}\label{localexistencelemma}
    Let $\mathcal{I} = (g_0,k,H_0,h_0,\xi_0,x_0)$ be string frame initial data on $\Sigma$. Let $U_\Sigma \subset \Sigma$ be a coordinate neighbourhood of $p \in \Sigma$ with coordinates $(x^1,...,x^n)$, and define coordinates $(x^0 = t,x^1,...,x^n)$ on $U = \reals \times U_\Sigma$.
    
    Then, for every open neighbourhood $O \subset U$ of $p$, there exists an open neighbourhood $W \subset O$ of $p$ on which the string frame GEE (\ref{combinedgeneinsteineqs}) admit a globally hyperbolic development $(g,H,\xi)$ of the initial data on $W \cap \Sigma$. Furthermore, the development can be assumed to be such that the following properties are satisfied.
        \begin{enumerate}[label = (\arabic*)]
        \item There exists a dilaton potential $\phi$, $\xi = \extd \phi$, and thus an associated Einstein frame development $(\tilde{g},H,\phi)$ and Einstein frame initial data $\tilde{\mathcal{I}}$. There also exists a $B$-field potential, $H= \bar{H}_{\tilde{\mathcal{I}}} + \extd B$. These potentials satisfy the gauge conditions $\mathcal{D}[\tilde{g},\bar{g}]= 0$ and $\extd \mathcal{C}[\tilde{g},B,\bar{g}] = 0$ as well as the initial conditions from Lemmas \ref{initialdata_setup_metric_phi_lemma} and \ref{initialdata_setup_bfield_lemma} with an initial vanishing condition for $B$.
        \item The component-matrix $(\tilde{g}_{\mu\nu})_{\mu\nu}$ (or equivalently $(g_{\mu\nu})_{\mu\nu}$) takes values in the space of canonical Lorentz matrices $\mathcal{C}_n$. Then, in particular, $\grad t$ is timelike on $W$.
    \end{enumerate}
\end{lemma}
\begin{proof}
    We take an open contractible coordinate neighbourhood $V_\Sigma \subset \Sigma$ of $p$, so that we can write $\xi_0 = \extd^\Sigma \phi_0$ for some $\phi_0 \in C^\infty(V_\Sigma)$. We consider as initial data $(g_0,k,H_0,h_0,\phi_0,\phi_1)$ with $\phi_1 = x_0$. We transform this data to Einstein frame initial data $\tilde{\mathcal{I}} = (\tilde{g}_0,\tilde{k},\tilde{H}_0,\tilde{h}_0, \tilde{\phi}_0, \tilde{\phi}_1)$ according to the formulas from Lemma \ref{initialdata_einsteinframe_lemma}. We define on $\reals \times V_\Sigma$ the background fields $\bar{g}_{\tilde{\mathcal{I}}}$ and $\bar{H}_{\tilde{\mathcal{I}}}$ according to the formulas (\ref{backgroundfieldsdefeq}). Due to Proposition \ref{localformprop}, we can apply Proposition \ref{localexistenceprop_preliminaries} to obtain a globally hyperbolic development $(W,\tilde{g},B,\phi)$ of the Einstein frame initial data under the modified Einstein frame system (\ref{modifiedcombinedeinsteineqs}) with respect to the background fields $\bar{H}_{\tilde{\mathcal{I}}}$ and $\bar{g}_{\tilde{\mathcal{I}}}$. Note that Proposition \ref{localexistenceprop_preliminaries} allows us to impose on the development that, one, the component-matrix $(\tilde{g}_{\mu\nu})_{\mu\nu}$ takes values in the space of canonical Lorentz matrices $\mathcal{C}_n$ and, two, it satisfies on $W \cap \Sigma$ the initial conditions from Lemmas \ref{initialdata_setup_metric_phi_lemma} and \ref{initialdata_setup_bfield_lemma} with initial vanishing condition for the $B$-field. In particular, $\restr{\mathcal{D}[\tilde{g},\bar{g}]}{W\cap\Sigma} = 0$.
    
    By construction, the development $(\tilde{g}, B, \phi)$ is such that Lemma \ref{gaugefieldinitialvanishinglemma} applies. It follows that $\restr{\nabla \mathcal{D}[\tilde{g},\bar{g}]}{W\cap\Sigma} = 0$ and $\restr{\extd\mathcal{C}[\tilde{g}, B, \bar{g}]}{W\cap\Sigma} =0$. From Proposition \ref{bfieldgaugefield_sourcefreewaveeq_prop}, we see that the gauge conditions $\mathcal{D} = 0$ and $\extd \mathcal{C} = 0$ are implemented on $W$. Therefore, with $H =\bar{H}_{\tilde{\mathcal{I}}} + \extd B$, $(W,\tilde{g},H,\phi)$ solves the unmodified Einstein frame system (\ref{einsteinframecombinedeinsteineqs}). Noting that conformal transformations preserve global hyperbolicity, we see that a transformation of $(\tilde{g},H,\phi)$ to the string frame yields the desired development $(g, H, \xi)$, where $\xi = \extd \phi$.
\end{proof}
The next Lemma, a local uniqueness result in our setting, is built on a local uniqueness result established (implicitly) in the proof of  \cite[Theorem 14.2.]{cauchy}, presented in the preliminary Section \ref{existence_ghd_preliminaries} as Proposition \ref{uniquesoldomainprop_preliminaries}.
\begin{lemma}\label{uniquesoldomainlemma}
    Let $\mathcal{I}$ be string frame initial data on $\Sigma$. For $i=1,2$, let $W_i \subset O_i \subset M = \reals \times \Sigma$ be open subsets such that $W_1 \cap W_2 \cap \Sigma \neq \emptyset$ is simply connected, and denote $\Sigma_i \coloneqq W_i \cap \Sigma$. Assume that we have coordinates on $O_i$ that are adapted to $\Sigma$. Assume also that on $W_i$ we have string frame developments $(g_i,H_i,\xi_i)$ of the initial data satisfying conditions (1) and (2) from Lemma \ref{localexistencelemma}. Denoting by $\phi_i$ and $B_i$ the associated potentials, and by $\tilde{g}_i$ the associated Einstein frame metrics, assume that we have smooth continuations of $(\tilde{g}_i,B_i,\phi_i)$ to $O_i$, and assume that $O_i$ is convex with respect to $\tilde{g}_i$ and that $\Sigma_i$ is spacelike Cauchy in $(W_i,\tilde{g}_i)$. Assume finally that $\closure{W}_i$ is compact and contained in $O_i$, $i=1,2$.
    
    Then $(g_1,H_1,\xi_1) = (g_2,H_2,\xi_2)$ on the whole intersection $W_1 \cap W_2$.
\end{lemma}
\begin{proof}
    By assumption, $\restr{\xi_1}{\Sigma_1 \cap \Sigma_2} =\restr{\xi_2}{\Sigma_1 \cap \Sigma_2} $. We define $c = \restr{[\phi_1 - \phi_2]}{\Sigma_1 \cap \Sigma_2}$. Then $\extd^\Sigma c = \restr{[\xi_1 - \xi_2]}{\Sigma_1\cap\Sigma_2} = 0$, hence we can view $c$ as a real number. We set $\hat{\phi}_1 = \phi_1$ and $\hat{\phi}_2 = \phi_2 + c$, and denote by $\hat{g}_i$ the Einstein frame metric associated to the solution $(g_i,B_i,\hat{\phi}_i)$. 
    
    Due to the invariance of the initial conditions from Lemmas \ref{initialdata_setup_metric_phi_lemma} and \ref{initialdata_setup_bfield_lemma} under a change of Einstein frame, cf.\ Lemma \ref{dilaton_gaugetrafo_lemma}, we see that $(\hat{g}_1,B_1,\hat{\phi}_1)$ and $(\hat{g}_2,B_2,\hat{\phi}_2)$ take the same initial values on $\Sigma_1 \cap \Sigma_2$. Due to the invariance of the gauge conditions $\mathcal{D} = 0$ and $\extd\mathcal{C} = 0$ under a change of Einstein frame,  cf.\ Lemma \ref{dilaton_gaugetrafo_lemma}, $(\hat{g}_1,B_1,\hat{\phi}_1)$ and $(\hat{g}_2,B_2,\hat{\phi}_2)$ are both solutions of the modified GEE (\ref{modifiedcombinedeinsteineqs}). Because conformal transformations with constant conformal factor leave geodesics invariant, $(O_2,\hat{g}_2)$ is convex. Thus, we are in the setting to apply Proposition \ref{uniquesoldomainprop_preliminaries}. The result follows. 
\end{proof}
With these results, we establish the existence of a globally hyperbolic development in our setting, following again the ideas of \cite[Theorem 14.2.]{cauchy}.
\begin{theorem}\label{genEinsteinlocalexistencethm}
    Let $(\Sigma, g_0, k, H_0, h_0, \xi_0, x_0)$ be initial data for the string frame GEE (\ref{combinedgeneinsteineqs}). Then there exists a globally hyperbolic string frame development of the data.
\end{theorem}
\begin{proof}
    From Lemma \ref{localexistencelemma}, we know that around every point $p \in \Sigma$ we can construct a local string frame development with an associated Einstein frame development satisfying initial and gauge conditions determining it uniquely, and which takes metric values in $\mathcal{C}_n$. To patch the local solutions together, we associate to every point $p \in \Sigma$ multiple neighbourhoods so as to arrive in the setting of Lemma \ref{uniquesoldomainlemma}.
    
    First, let $(U^\Sigma_p,x^m_p)$ be a coordinate neighbourhood of $p$ in $\Sigma$, so that we obtain coordinates $x_p^\mu$ on $\reals \times U^\Sigma_p \subset M$, where $x_p^0 = t$. Then, we choose $W'_p \subset \reals \times U^\Sigma_p$ with string frame development $(g_p, H_p, \xi_p)$ and associated Einstein frame development $(\tilde{g}_p,H_p,\phi_p)$ as provided by Lemma \ref{localexistencelemma}. In particular, we ask the development to satisfy the additional demands (1) and (2) from Lemma \ref{localexistencelemma}. We choose $O_p \subset W_p'$ convex with respect to $\tilde{g}_p$. Finally, we choose an open subset $W_p \subset O_p$ with the properties that, one, its compact closure is contained in $O_p$, two, $(W_p,\tilde{g}_p)$ has $\Sigma_p \coloneqq W_p \cap \Sigma$ as Cauchy hypersurface, three, $(\Sigma_q)_{q\in\Sigma}$ is a good open cover of $\Sigma$ in the sense that the intersection of any two of its members is simply connected.
    
    From the local data above, one can patch together a development on $D \coloneqq \cup_{p\in \Sigma} W_p$. By assumption, $\Sigma_p\cap\Sigma_q$ is simply connected for all $p,q \in \Sigma$, and conditions (1) and (2) from Lemma \ref{localexistencelemma} hold. Hence, it follows from Lemma \ref{uniquesoldomainlemma} that the constructed solutions agree on $W_p \cap W_q$. We obtain a well-defined string frame solution $(g,H,\xi)$ on $D$. By Lemma \ref{gh_union_lemma_preliminaries}, because for all $p \in \Sigma$ the local development $(W_p,g_p)$ is globally hyperbolic and $g_p$ takes values in $\mathcal{C}_n$, $(D,g)$ is globally hyperbolic with Cauchy hypersurface $\Sigma$. Thus, with the embedding $i \colon \Sigma \hookrightarrow M$ given by $p \mapsto (0,p)$, we have that $(g,H,\xi)$ is a globally hyperbolic string frame development of the initial data $(g_0,k,H_0,h_0,\xi_0,x_0)$ on $\Sigma$. 
\end{proof}

\subsection{Local Geometric Uniqueness}\label{uniqueness_ghd_section}

In this section, we prove local geometric uniqueness of string frame developments, following the proof strategy of \cite[Theorem 14.3]{cauchy}. That is, we show that any two developments are extensions of a common development. The proof relies on comparing any given development $(M',g',H',\xi')$ with the development $(D,g,H,\xi)$ constructed in Theorem \ref{genEinsteinlocalexistencethm}. 

In the construction of the development on $D$, one defines in a local Einstein frame background fields $\bar{g}$ and $\bar{H}$ (cf.\ (\ref{backgroundfieldsdefeq})), and decomposes $H = \bar{H} + \extd B$. Crucially, $(g,B,\phi)$ are constructed such that they satisfy the gauge conditions $\mathcal{D} = 0$ and $\extd\mathcal{C}=0$ and hence solve the modified system (\ref{modifiedcombinedeinsteineqs}) whose solutions we know to be locally unique by Proposition \ref{uniquesoldomainprop_preliminaries}. Thus, the idea is to locally relate $(M',g',H',\phi')$ to $(D,g,H,\phi)$ by a diffeomorphism $f\colon D \supset W \to W' \subset M'$ and then decompose $f^*H' = \bar{H} + \extd\hat{B}'$ such that \[ (\hat{g}',\hat{B}',\hat{\phi}') = (f^*g,\hat{B}',f^*\phi)\]
implements the gauge conditions $\mathcal{D}[\hat{g}',\bar{g}] =0$ and $\extd\mathcal{C}[\hat{g}',\hat{B}',\bar{g}]=0$. Then, from the aforementioned uniqueness property, we can conclude that our two solutions coincide. In a final step, one has to combine the local diffeomorphisms to a global one.

Before we carry out the full construction, we focus on the existence of a diffeomorphism $f$ and a decomposition $f^*H' = \bar{H} + \extd \hat{B}'$ as above. 

Recall that by Proposition \ref{deturckgauge_prop_preliminaries}, one can locally find a diffeomorphism $f$ implementing the DeTurck gauge condition $\mathcal{D}[f^*g,\bar{g}]$, as well as the initial conditions on the metric discussed in Lemma \ref{initialdata_setup_metric_phi_lemma}. We now prove a similar result for the $B$-field. That is, we show that by a closed $B$-field transformation, one can implement on any given $B$-field the generalised Lorenz gauge $\extd\mathcal{C}=0$ as well as the initial conditions from Lemma \ref{initialdata_setup_bfield_lemma}.
\begin{lemma}\label{bfieldgaugelemma}
    Let $(M,g)$ be a globally hyperbolic smooth spacetime, $B\in \Omega^2(M)$ a two-form, and $\Sigma$ a smooth spacelike hypersurface.  Denote the future-pointing unit normal on $\Sigma$ by $N$. Let $\bar{g}$ be a second Lorentzian metric on $M$, the background metric. Let $B_0 \in \Omega^2(\Sigma)$ be a two-form on $\Sigma$ such that $\extd^\Sigma(B^\parallel - B_0) = 0$, and denote $B_1 = (\extd B)(N) -  k\cdot B_0 \in \Omega^2(\Sigma)$.
    
    Then, there exists a closed two-form $\beta \in \Omega^2_{\mathrm{cl}}(M)$ such that the gauge transformed quantity $\hat{B} = B + \beta$ satisfies the generalised Lorenz gauge $\extd\mathcal{C}[g,\hat{B}] = 0$ and the initial conditions from Lemma \ref{initialdata_setup_bfield_lemma}.
\end{lemma}
\begin{proof}
    Let us define for some $\beta \in \Omega^2_{\mathrm{cl}}(M)$
    \begin{equation*}
        \hat{B} \coloneqq B + \beta, \qquad\quad \hat{\mathcal{C}} \coloneqq -\extd^*_{g,\Bar{g}} \hat{B}
    \end{equation*}
    Then $\extd\hat{\mathcal{C}} = 0$ is equivalent to
    \begin{equation}\label{gaugeconditioneq}
        \begin{split}
            0 &= \extd \extd^*_{g,\Bar{g}} \hat{B} = \extd \extd^*_{g,\Bar{g}}  B + \extd \extd^*_{g,\Bar{g}} \beta \\
            & =  \extd \extd^* \beta + \extd [F(\beta)] + \extd\extd^*_{g,\Bar{g}}  B
        \end{split}
    \end{equation}
    where $F\in \Gamma(\Hom(\Lambda^2 T^*M, T^*M))$, explicitly $F(\beta) = (\extd^*_{g,\Bar{g}}- \extd^*)\beta$. We claim that we obtain a solution to (\ref{gaugeconditioneq}), if we define $\beta$ to be, as provided by Theorem \ref{tensorwaveeqthm}, a solution to the following equation:
    \begin{equation*}
        0 = \Box_{\mathrm{Hd}}\beta - \extd [F(\beta)] - \extd\extd^*_{g,\Bar{g}}  B
    \end{equation*}
    Assume that we manage to set up initial conditions for $\beta$ such that $\extd \beta = 0$ on $\Sigma$. Taking the exterior derivative of this equation, we then find that
    \begin{equation*}
        0 = -\extd \extd^*\extd \beta = \Box_{\mathrm{Hd}} \extd \beta
    \end{equation*}
    By Theorem \ref{tensorwaveeqthm}, $\extd \beta \equiv 0$ globally, so that indeed $\beta$ solves \ref{gaugeconditioneq}.
    
    Finally, we note that by Lemma \ref{initialdata_setup_bfield_lemma}, the initial values of $\hat{B}$ and its time-derivative are uniquely determined by the requirements in the statement. The same then holds for $\beta$. It remains to check that this is consistent with the requirement $\extd \beta = 0$. Clearly, $(\extd\beta)^\parallel = \extd^\Sigma(\hat{B}-B)^\parallel = 0$. Also, by Lemma \ref{extdhypersurflemma} and definition of $B_1$, 
    \begin{equation*}
        (\extd \hat{B})(N) = B_1 + k \cdot B_0 = (\extd B)(N),
    \end{equation*}
    so that also $(\extd \beta)(N) = 0$.
\end{proof}
Combining the Proposition \ref{deturckgauge_prop_preliminaries} and Lemma \ref{bfieldgaugelemma}, we can prove that any given Einstein frame development of initial data on a hypersurface $\Sigma$ can be locally related to a solution on $M = \reals \times \Sigma$ with background fields as defined in (\ref{backgroundfieldsdefeq}) such that the gauge conditions $\mathcal{D} = 0$ and $\extd\mathcal{C} = 0$ as well as the conditions on the initial data from Lemmas \ref{initialdata_setup_metric_phi_lemma} and \ref{initialdata_setup_bfield_lemma} are implemented.
\begin{proposition}\label{combinedgauge_prop}
    Let $\tilde{\mathcal{I}}=(\Sigma, \tilde{g}_0, \tilde{k}, H_0, \tilde{h}_0, \phi_0, \tilde{\phi}_1)$ be initial data for the Einstein frame GEE (\ref{combinedgeneinsteineqs}). Let $(M', \tilde{g}', H', \phi')$ be an arbitrary Einstein frame development. Denote the associated embedding by $i \colon \Sigma \hookrightarrow M'$. Denote furthermore $M = \reals \times \Sigma$ and by $\bar{g} = \bar{g}_{\tilde{\mathcal{I}}}$ and $\bar{H} = \bar{H}_{\tilde{\mathcal{I}}}$ the background fields defined in (\ref{backgroundfieldsdefeq}). 

    Then, for every $p\in \Sigma$ there exist neighbourhoods $p \in W \subset M$ and $i(p) \in W' \subset M'$ and a diffeomorphism $f \colon W \to W'$ such that with $\hat{g}' = f^*\tilde{g}'$, $\hat{H}' = f^*H'$, and $\hat{\phi}' = f^*\phi'$, one can decompose $\hat{H}' = \bar{H} + \extd \hat{B}'$ such that  the gauge conditions $\mathcal{D}[\hat{g}',\bar{g}] = 0$ and $\extd\mathcal{C}[\hat{g}',\hat{B}',\bar{g}] = 0$ hold and the initial conditions from Lemmas \ref{initialdata_setup_metric_phi_lemma} and \ref{initialdata_setup_bfield_lemma} are satisfied with initial vanishing condition for the $B$-field.
\end{proposition}
\begin{proof}
    Take $p \in \Sigma$ and a diffeomorphism $f\colon W \to W'$ such that $\mathcal{D}[f^*\tilde{g}',\bar{g}]= 0$, where $p \in W \subset M$. By Proposition \ref{deturckgauge_prop_preliminaries}, such a diffeomorphism exists and can be chosen such that $f_* \partial_t = N'$ and $\restr{f_*}{TW_\Sigma} = i_{*}$, where we set $W_\Sigma = W \cap \Sigma$ and denoted by $N'$ the future-pointing unit-normal on $i(\Sigma)$. By restricting, we can assume that $(W',\tilde{g}')$ is globally hyperbolic with Cauchy hypersurface $W' \cap \Sigma$. 
    
    Set $\hat{g}' = f^*\tilde{g}'$, $\hat{H}' = f^* H'$, and $\hat{\phi}' = f^*\phi'$. Note that the pullback of the data $(\tilde{g}',H',\phi')$ by $f$ induces on $f^{-1}\circ i(W_\Sigma) = W_\Sigma$ the pullback of the initial data $(\tilde{g}_0,\tilde{k},H_0,h_0,\phi_0,\phi_1)$ by $i^*\circ (f^{-1})^* = \idon{TW_\Sigma}$, so the initial data itself. Since also $f_*N = N'$, we see that $(\hat{g}',\hat{\phi}')$ satisfies the conditions from Lemma \ref{initialdata_setup_metric_phi_lemma}.
    
    We want to define a two-form $\hat{B}'$ that satisfies the gauge condition $\extd \mathcal{C}[\hat{g}', \hat{B}'] = 0$ and the initial conditions from Lemma \ref{initialdata_setup_bfield_lemma}. First, consider that, since $W \subset M$, $\bar{H}$ is defined on $W$. We can assume $W$ to be contractible, so that all closed forms are exact. Hence, we can find a two-form $B' \in \Omega^2(W)$ such that $\hat{H}' = \bar{H} + \extd B'$. Global hyperbolicity of $(W',\tilde{g}')$ is equivalent to global hyperbolicity of $(W,\hat{g}')$. Hence, by Lemma \ref{bfieldgaugelemma}, we can find a closed two-form $\beta'$ such that the gauge-transformed quantity $\hat{B}' \coloneqq B' + \beta'$ satisfies the gauge-condition $\extd \mathcal{C}[\hat{g}',\hat{B}'] = 0$ and the initial conditions from Lemma \ref{initialdata_setup_bfield_lemma} with $B_0 = B_1 = 0$. Note that this choice is allowed because, by $f_*N = N'$ and agreement of the initial data induced by $H$ and $\hat{H}'$, $\restr{\hat{H}'}{\Sigma} = \restr{H}{\Sigma}$ and thus $\restr{\extd B'}{\Sigma} = 0$. This finishes the proof.
\end{proof}

Finally, we can adapt \cite[Theorem 14.3]{cauchy} to our setting, following closely the proof strategy presented there. We show that any two developments of the initial data are extensions of a common development by showing that any development is an extension of a development as constructed in Theorem \ref{genEinsteinlocaluniquenessthm}.
\begin{theorem}\label{genEinsteinlocaluniquenessthm}
    Let $(\Sigma, g_0, k, H_0, h_0, \xi_0, x_0)$ be initial data for the string frame equations (\ref{combinedgeneinsteineqs}). Denote by $(D, g, H, \xi)$ a globally hyperbolic development constructed as in Theorem \ref{genEinsteinlocalexistencethm}, $D\subset M = \reals \times \Sigma$. Assume that we have another development $(M', g', H', \xi')$ with embedding $i \colon \Sigma \hookrightarrow M'$. 
    
    Then, there is a tubular neighbourhood $\tilde{D} \subset D$ of $\Sigma$ and a smooth time orientation preserving diffeomorphism $\psi \colon \tilde{D} \to \psi(\tilde{D}) \subset M'$ with $ \restr{\psi}{\Sigma} = i$ relating the two solutions, i.e.\ $\psi^*g' = g$, $\psi^* H' = H$, and $\psi^* \xi' = \xi$.
\end{theorem}
\begin{proof}
    Recall that in the construction performed in Theorem \ref{genEinsteinlocalexistencethm}, we associate to every point $p \in \Sigma$ a coordinate neighbourhood $U_p = \reals \times U_p^\Sigma$ with coordinates $(x_p^0 = t,x_p^m)$ and then an open subset $W_p' \subset U_p$ on which there is associated to the string frame development a globally hyperbolic Einstein frame development $(\tilde{g}_p, H_p, \phi_p)$ satisfying conditions (1) and (2) from Lemma \ref{localexistencelemma}. In particular, we have background fields $\bar{g}_p$ and $\bar{H}_p$ on $W_p'$. We also picked a $\tilde{g}_p$-convex subset $p \in O_p \subset W_p'$ and a globally hyperbolic neighbourhood $(W_p, \tilde{g}_p)$ with compact closure contained in $O_p$.
    \begin{center}
        \label{geom_uniqueness_fig}
        \captionsetup{type=figure}
        \includegraphics[width = 1.0\textwidth]{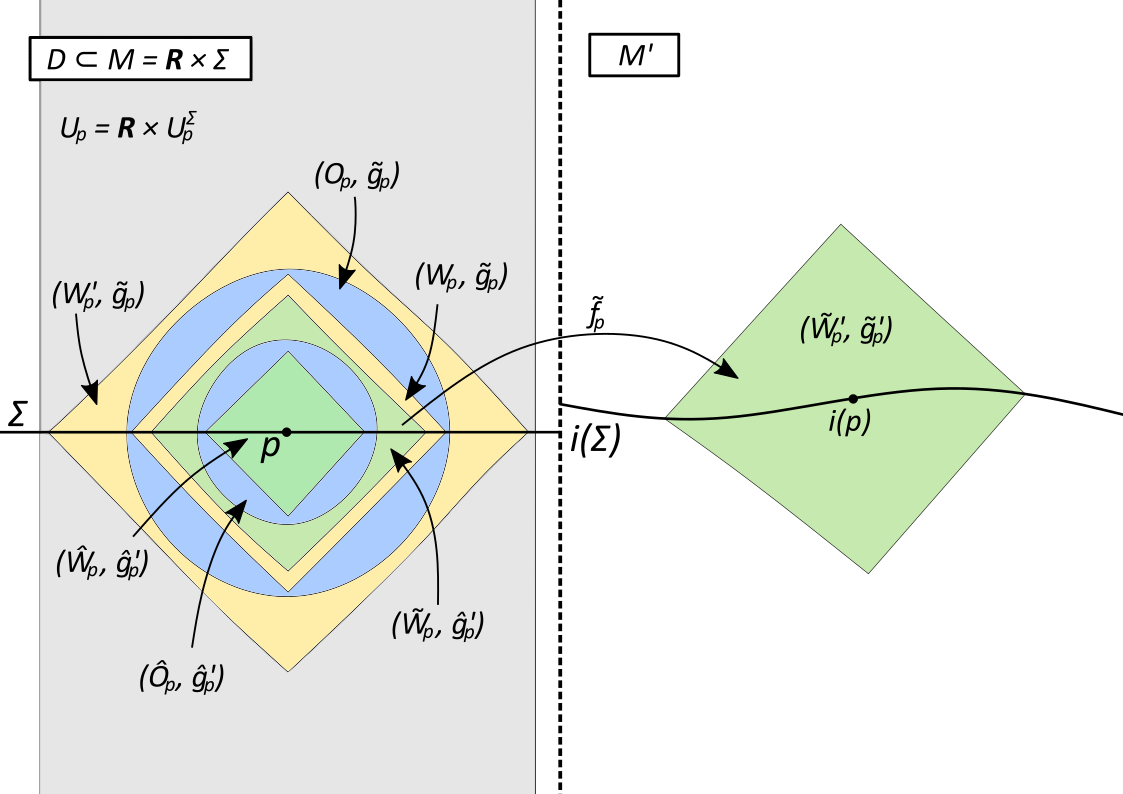}
        \captionof{figure}{An illustration of the data associated to a point $p$ of the initial hypersurface $\Sigma$ in the proof of Theorem \ref{genEinsteinlocaluniquenessthm}. On the left hand side is the gauge-fixed development $D$ of the initial data constructed in Theorem \ref{genEinsteinlocalexistencethm}. On the right hand side is an arbitrary development $M'$. The two developments are compared (in a local Einstein frame) via a diffeomorphism $\tilde{f}_p\colon \widetilde{W}_p \to \widetilde{W}'_p$.}
    \end{center}
    Let us associate additional data to every point $p \in \Sigma$. We choose a diffeomorphism $\widetilde{f}_p \colon \widetilde{W}_p \to \widetilde{W}_{p}'$ defined as in Proposition \ref{combinedgauge_prop} with the additional requirement that $\widetilde{W}_p \subset W_p$. We obtain the gauge-fixed solution $(\widehat{g}'_{p},\widehat{B}'_{p},\widehat{\phi}'_{p})$ as in that Proposition. We pick a $\widehat{g}_p'$-convex neighbourhood $\widehat{O}_{p} \subset \widetilde{W}_p$ and a subset $\widehat{W}_p \subset \widehat{O}_p$ such that, one, $(\widehat{W}_p,\widehat{g}_p')$ is globally hyperbolic with Cauchy hypersurface $\widehat{W}_p \cap \Sigma$, two, $\widehat{W}_p$ has compact closure contained in $\widehat{O}_p$, three, $\widehat{g}_p'$ takes on $\widehat{W}_p$ values in $\mathcal{C}_n$ and, four, $\widehat{W}_p \cap \Sigma$ is simply-connected. We define $f_p\coloneqq \restr{\widetilde{f}_p}{\widehat{W}_p}$. 
    
    By their construction, both $(W_p, \tilde{g}_p, B_p, \phi_p)$ and $(\widehat{W}_p, \widehat{g}_p', \widehat{B}_p',\widehat{\phi}_p')$ satisfy (1) and (2) from Lemma \ref{localexistencelemma}. Also, the intersection $W_p \cap \widehat{W}_p \cap \Sigma = \widehat{W}_p \cap \Sigma$ is simply connected. Thus, we are in the setting to apply Lemma \ref{uniquesoldomainlemma}. We conclude that $\restr{(g,H,\xi)}{\widehat{W}_p} = (f_p^*g',f_p^*H',f_p^*\xi')$, so $f_p$ is an isometry relating the string frame solutions locally.  

    To finish the proof, we have to patch the local isometries together to a global isometry. For every $p \in \Sigma$, we find a neighbourhood $\tilde{D}_p\subset W_p$ such that $(\tilde{D}_p,g)$ is globally hyperbolic with Cauchy HS $\tilde{D}_p \cap \Sigma$. Then, by Lemma \ref{intersection_gh_lemma_preliminaries}, the isometries $f_p$ and $f_q$ are determined on $\tilde{D}_p \cap \tilde{D}_q$ by the value of their differential on $\Sigma \cap \tilde{D}_p \cap \tilde{D}_q$. These agree, and hence $\restr{f_p}{\tilde{D}_p \cap \tilde{D}_q} = \restr{f_q}{\tilde{D}_p \cap \tilde{D}_q}$. We conclude that we can define a global isometry $\psi$ on $\tilde{D}= \bigcup_p \tilde{D}_p$, as desired.
\end{proof}

\subsection{The Maximal Globally Hyperbolic Development}\label{mghd_section}

In the famous work \cite{choquet1969global}, Choquet-Bruhat and Geroch established for the Einstein equations the existence of a geometrically unique globally hyperbolic development which extends any other development of a given set of initial data.\footnote{We note that many details absent from the proof in \cite{choquet1969global} are provided by Ringström in \cite[§ 23]{topology_stability_universe_ringström} in the context of the Einstein-Vlasov-nonlinear scalar field system. A text similar to \cite{topology_stability_universe_ringström}, but in the context of the Einstein-nonlinear scalar field system, is provided in Ringström's errata \cite{cauchy_errata} to his book \cite{cauchy} (which contains an erroneous proof of the existence of an MGHD). A proof that does not (in contrast to the other mentioned proofs) rely on Zorn's Lemma was given by Sbierski \cite{sbierski2016existence}.} We formalise the notion of the maximal globally hyerbolic development for the string frame GEE as follows (we adapt \cite[Definition 16.5.]{cauchy}).
\begin{definition}\label{mghd_def}
    Let $\mathcal{I} = (\Sigma,g_0,k,H_0,h_0,\xi_0,x_0)$ be initial data for the string frame GEE. A \textit{maximal globally hyperbolic development (MGHD)} of $\mathcal{I}$ is a development $(M,g,H,\xi)$ with embedding $i\colon \Sigma \hookrightarrow M$ such that for every other globally hyperbolic development $(M',g',H',\xi')$ with embedding $i'\colon \Sigma\hookrightarrow M'$ there exists a map $\psi \colon M' \to \psi(M')\subset M$ which is a time-orientation preserving diffeomorphism onto its image relating the two developments, i.e. $\psi^*g = g'$, $\psi^*H = H'$, $\psi^*\xi = \xi'$, and $\psi \circ i' = i$.
\end{definition}
If $(M,g,H,\xi)$ and $(M',g',H',\xi')$ are two developments of given initial data for which a map $\psi$ as in the definition exists, we also call $(M,g,H,\xi)$ an \textit{extension} of $(M',g',H',\xi')$. Thus, an MGHD $(M,g,H,\xi)$ is an extension of every other globally hyperbolic development.

It follows from the properties of the MGHD that it is unique up to diffeomorphism, i.e.\ for two MGHDs $(M,g,H,\xi)$ and $(M',g',H',\xi')$ one can see any map $\psi$ as provided by Definition \ref{mghd_def} to be a diffeomorphism between $M$ and $M'$. \cite{cauchy}

The explicit proof of existence for the MGHD given in \cite{choquet1969global} focuses on the Einstein vacuum system. Remarks preceding the proof explain\footnote{The same remarks apply to the more detailed proof due to Ringström mentioned in a previous footnote.} why the results extend to general Einstein matter systems for which the notions of solutions, initial data, developments, and extensions are as in (1-4) and satisfy the properties (i-iii).
\begin{enumerate}
    \item[(1)] A \textit{solution} is a tuple $(M,g,\Phi)$ consisting of an $n+1$-dimensional time-oriented Lorentzian manifold $(M,g)$, a section $\Phi \in \Gamma(E)$ of a diffeomorphism-invariant vector subbundle $E$ of the tensor bundle $\mathcal{T}M$, all such that a diffeomorphism-invariant set of conditions (the Einstein-matter equations) is satisfied.
    \item[(2)] \textit{Initial data} is a tuple $(\Sigma,g_0,k,\Phi_0)$ consisting of an $n$-dimensional Riemannian manifold $(\Sigma,g_0)$, a symmetric two-tensor $k$, and a section $\Phi_0 \in \Gamma(E_0)$ of a diffeomorphism-invariant vector subbundle of the tensor bundle $\mathcal{T}\Sigma$ satisfying a given diffeomorphism-invariant set of constraint conditions.\footnote{The constraint conditions do not have to be of a specific form; their purpose is to encode conditions necessary and sufficient to guarantee property (ii).}
    \item[(i)] A solution $(M,g,\Phi)$ naturally induces on every spacelike hypersurface $\Sigma$ a set of initial data $\mathcal{I} = (\Sigma,g_0,k,\Phi_0)$, where $g_0$ is the inherited Riemannian metric and $k$ the second fundamental form on $\Sigma$. Naturality means that, if $\psi \colon M' \to M$ is a diffeomorphism and $i\colon \Sigma \hookrightarrow M$ the embedding associated to $\mathcal{I}$, then $(M',\psi^*g,\psi^*\Phi)$ induces on $\psi(i(\Sigma))$ the initial data $(\psi(i(\Sigma)),(\psi\circ i)_*g_0,(\psi\circ i)_*k, (\psi\circ i)_*\Phi_0)$.
    \item[(3)] A \textit{development} of initial data $(\Sigma,g_0,k,\Phi_0)$ is a solution $(M,g,\Phi)$ and a Riemannian embedding $i\colon \Sigma \hookrightarrow M$ such that the pullback by $i$ of the initial data induced on $i(\Sigma)$ yields the initial data $(\Sigma,g_0,k,\Phi_0)$.
    \item[(4)] Given initial data $\mathcal{I}$, a development $(M,g,\Phi)$ of $\mathcal{I}$ with embedding $i$ is an \textit{extension} of another development $(M',g',\Phi')$ of $\mathcal{I}$ with embedding $i'$ if there exists a map $\psi\colon M' \to \psi(M') \subset M$ which is a time-orientation preserving diffeomorphism onto its image such that $\psi^*g = g'$, $\psi^*\Phi = \Phi'$, and $\psi\circ i' = i$.
    \item[(ii)] Every set of initial data admits a globally hyperbolic development.
    \item[(iii)] Any two developments of given initial data are extensions of a common development.
\end{enumerate}

We see that our notions of string frame solutions, string frame initial data (see Definition \ref{initialdata_stringframe_def}), string frame developments (see Definition \ref{development_stringframe_def}), and extensions (see above) fit the template (1-4). In our setting, property (i) is trivial, and properties (ii) and (iii) respectively correspond to Theorems \ref{genEinsteinlocalexistencethm} and \ref{genEinsteinlocaluniquenessthm}. Thus,
\begin{theorem}\label{mghd_thm}
    Let $\mathcal{I}$ be initial data for the string frame GEE. Then there exists a string frame MGHD of $\mathcal{I}$. It is unique up to diffeomorphism.
\end{theorem}
\cleardoublepage
\appendix
\section{Results on Wave Equations}

We summarise here the results on symmetric hyperbolic PDEs, also called wave equations, which we employ in our study of the initial value problem for the generalised Einstein equations. 
\subsection{Non-Linear Wave Equations}\label{nonlinear_waveeqs_appendix}

In this section, we give appropriate definitions for the discussion of and the main result on non-linear wave equations. Specifically, we discuss systems of PDEs of the form (\ref{hyperbolicpdeeq}). Recall that this is a system of PDEs for a vector valued function $u \colon \reals^{n+1} \to \reals^N$, where $n,N \in \mathbb{N}_{\geq 1}$.

Working on $\reals^{n+1}$, we employ the following notion of a canonical Lorentz matrix, cf.\ \cite[Definition 8.4.]{cauchy}. 
\begin{definition}
    Let $A \in \Sym^2(n+1)$ be a symmetric matrix. We call $A$ a \textit{canonical Lorentz matrix} if $A_{00} < 0$ and $(A_{ij})$ positive definite. We denote the space of canonical Lorentz matrices by $\mathcal{C}_n$. Finally, given $a = (a_0,a_1,a_2) \in \reals_{>0}^3$, we denote by $\mathcal{C}_{n,a}$ the space of canonical Lorentz matrices $A$ such that $A_{00} < - a_0$, $(A_{ij}) > a_1$, and $\norm{A}_1 > a_2$, where $\norm{\cdot}_1$ denotes the $1$-norm on $\reals^{(n+1)^2}$.
\end{definition}
One can show that a canonical Lorentz matrix $A$ has $n$ positive eigenvalues and one negative eigenvalue \cite[Lemma 8.3.]{cauchy}.

Recall now more specifically that (\ref{hyperbolicpdeeq}) is the following system of PDEs:
\begin{equation}\label{hyperbolicpdeeq_appendix}
    \begin{split}
        g[u]^{\mu\nu} \partial_\mu \partial_\nu u &= f[u], \\
        u(T_0, \cdot ) &= U_0, \\
        \partial_t u(T_0, \cdot) &= U_1.
    \end{split}
\end{equation}
Herein, $T_0$ some real value, $U_0, U_1 \in C_0^\infty(\reals^n,N)$ compactly supported initial data, and 
\[ g \in C^\infty(\reals^{N+(n+1)N+n+1},\mathcal{C}_n), \qquad\quad f \in C^\infty(\reals^{N+(n+1)N+n+1},\reals^N)\] 
respectively a $C^\infty$ ($N,n$)-admissible metric and a $C^\infty$ ($N,n$)-admissible non-linearity. Furthermore, we employed the notation
\begin{equation*}
    \begin{split}
        g[u](t,x) &\coloneqq g(t,x,u(t,x),\partial_0u(t,x),...,\partial_nu(t,x)), \\
        f[u](t,x) &\coloneqq g(t,x,u(t,x),\partial_0u(t,x),...,\partial_nu(t,x)).
    \end{split}
\end{equation*}
In particular, $g[u] \in C^\infty(\reals^{n+1},\mathcal{C}_n)$ and $f[u] \in C^\infty(\reals^{n+1}, \reals^N)$.

We explain now the two notions of a $C^\infty$ ($N,n$)-admissibility, cf.\ \cite[Definitions 9.1 and 9.4.]{cauchy}.
\begin{definition}
    Let $g \in C^\infty(\reals^{N+(n+1)N+n+1},\mathcal{C}_n)$. We call $g$ a \textit{$C^\infty$ ($N,n$)-admissible metric} if for all compact intervals $I \subset \reals$
    \begin{enumerate}[label = (\roman*)]
        \item there exists $a \in \reals_{>0}^3$ such that $g(t,\cdot) \in \mathcal{C}_{n,a}$ for all $t \in I$, and
        \item for all multi-indices $\alpha \in \mathbb{N}_0^{N+(n+1)N+n+1}$, there exists a continuous, monotonously increasing function $h \colon \reals \to \reals$ such that
        \begin{equation*}
            \norm{\partial^\alpha g(t,x,\xi)} \leq h(\norm{\xi}), \qquad t\in I, x \in \reals^n, \xi \in \reals^{N+(n+1)N},
        \end{equation*}
        where $\norm{\cdot}$ is the maximum norm on $\reals^k$ for appropriate $k$. 
    \end{enumerate}
\end{definition}
\begin{definition}
    Let $f \in C^\infty(\reals^{N+(n+1)N+n+1}, \reals^N)$. We call $f$ a \textit{$C^\infty$ ($N,n$)-admissible non-linearity}, if
    \begin{enumerate}[label=(\roman*)]
        \item $f[0]$ is of locally $x$-compact support, i.e.\ for all $t \in \reals$ $f[0](t)$ is of compact support, and
        \item for all compact intervals $I \subset \reals$ and all multi-indices $\alpha \in \mathbb{N}_0^{N+(n+1)N+n+1}$, there exists a continuous, monotonously increasing function $h \colon \reals \to \reals$ such that
        \begin{equation*}
            \norm{\partial^\alpha f(t,x,\xi)} \leq h(\norm{\xi}), \qquad t\in I, x \in \reals^n, \xi \in \reals^{N+(n+1)N},
        \end{equation*}
        where $\norm{\cdot}$ is the maximum norm on $\reals^k$ for appropriate $k$. 
    \end{enumerate}
\end{definition}

We can now cite \cite[Corollary 9.16.]{cauchy}, which for us is the main theorem on non-linear wave equations. Essentially verbatim, it states the following.
\begin{theorem}\label{uniqueexistencehyperbolicpdethm}
    Consider the system (\ref{hyperbolicpdeeq}), where $T_0 \in \mathbb{R}$, $g \in C^\infty(\reals^{N+(n+1)N+n+1},\mathcal{C}_n)$ a $C^\infty$ ($N,n$)-admissible metric, $f \in C^\infty(\reals^{N+(n+1)N+n+1},\reals^N)$ a $C^\infty$ ($N,n$)-admissible non-linearity, and $U_0, U_1 \in C_0^\infty(\reals^n,N)$ initial data.
    
    Then there are $T_1 < T_0 < T_2$ and a unique solution $U\in C^\infty[(T_1,T_2)\times \reals^n, \reals^N]$ to (\ref{hyperbolicpdeeq}). The solution is of locally $x$-compact support. Moreover, $T_2$ can be chosen such that $T_2 = \infty$ or
    \begin{equation*}
        \lim_{\tau \to T_2 -} \sup_{T_0\leq t \leq \tau} \sum_{\absolute{\alpha}+j \leq 2} \sup_{x\in\reals^n}\absolute{\partial^\alpha \partial^j_t u(t,x)} = \infty.
    \end{equation*}
    The statement concerning $T_1$ is similar.
\end{theorem}

\subsection{Inhomogeneous Linear Tensor Wave Equations}
Given a fixed smooth Lorentzian metric $g$ on the manifold $M$, we often deal with tensor wave equations. To discuss them, it is convenient to make the following
\begin{definition}
    Given a tensor $A \in \Gamma(T^r_sM)$, we define with with $r \geq l$, $s \geq k$ the following map:
    \begin{equation*}
        A\colon \Gamma(T^k_lM)\longrightarrow \Gamma(T^{r-l}_{s-k}M), \qquad\quad A(B)_{\beta_1...\beta_{s-k}}^{\alpha_1...\alpha_{r-l}} \coloneqq A_{\delta_1...\delta_k\beta_1...\beta_{s-k}}^{\gamma_1...\gamma_l\alpha_1...\alpha_{r-l}} B_{\gamma_1...\gamma_l}^{\delta_1...\delta_k}.
    \end{equation*}
    If $l<r$ and $s <k$, we also set $A(B) \coloneqq B(A)$.

    Furthermore, given a metric $g$ and two tensors $A \in \Gamma(T_sM)$, $B\in  \Gamma(T_l M)$, w.l.o.g. $s \geq l$, we denote
    \begin{equation*}
        \scalbrack{A, B}_{\beta_1...\beta_{s-l}} \equiv \scalbrack{B,A}_{\beta_1...\beta_{s-l}} \coloneqq B^{\gamma_1...\gamma_l} A_{\gamma_1...\gamma_l\beta_1...\beta_{s-l}} \in \Gamma(T_{s-l} M).
    \end{equation*}
\end{definition}
The main result on tensor wave equations that we employ is the following. It is a combination of \cite[Theorem 12.17. and Corollary 12.12]{cauchy}.
\begin{theorem}\label{tensorwaveeqthm}
    Let $(M,g)$ be a globally hyperbolic spacetime with spacelike Cauchy hypersurface $i\colon S\hookrightarrow M$. Assume $B \in T^{r+s+1}_{r+s}(M)$, $C\in T^{r+s}_{r+s}(M)$, $E\in T^r_s(M)$. Given smooth tensor fields $A_0,A_1 \in T^r_s(M)$, there exists a unique smooth tensor field $A\in T^r_s(M)$ solving the initial value problem
    \begin{equation}\label{tensorwaveeq}
        \begin{split}
            \Box_gA + B(\nabla A) + C(A) &= E, \\
            i^*A &= i^*A_0,\\
            i^*(\nabla_N A) &= i^*A_1,
        \end{split}
    \end{equation}
    where $\Box_g = \nabla^\lambda\nabla_\lambda$ and $N$ is the future directed normal to $S$.

    Furthermore, if $E \equiv 0$, and also $i^*A_0 = i^*A_1 = 0$ on $\Omega\subset S$, then $A = 0$ on the Cauchy development $D(\Omega)$.
\end{theorem}
Assuming $H = \extd B$, the $B$-field has to satisfy the equation (cf.\ (\ref{einsteinframecombinedeinsteineqs}) with $\Bar{H} = 0$)
\begin{equation*}
    \extd^*\extd B = -\frac{4}{d-2}i_{\xi^\sharp} \extd B
\end{equation*}
for some closed one-form $\xi \in \Omega^1_{\mathrm{cl}}(M)$. This almost looks like a tensor wave equation as above, but differs from it in two ways:
\begin{enumerate}[label = (\roman*)]
    \item The Beltrami wave operator $\Box_g$ is replaced by the co-exact wave operator $\extd^*\extd B$.
    \item The equation only depends on $\extd B$.
\end{enumerate}
Due to the second point in particular, we expect the PDE to exhibit fundamentally different uniqueness behaviour than the tensor wave equation (\ref{tensorwaveeq}): For any solution $B$ to such an equation and any closed one-form $b \in \Omega^1_{\mathrm{cl}}(M)$, also $B+ b$ is a solution. In physics, this is referred to as the $B$-field having a \enquote{gauge freedom}.

We will call equations that satisfy (i) and (ii) \textit{co-exact wave equations}. These are equations for a $p$-form $A \in \Omega^p(M)$ that can be stated as 
\begin{equation}\label{formwaveeq0}
    \extd^*\extd A+ B(\extd A) = E,
\end{equation}
wherein $B \in \Gamma(\Hom(\Lambda^{p+1}T^*M,\Lambda^pT^*M))$, and $E \in \Omega^p(M)$ is co-closed, i.e.\ $\extd^*E = 0$. 

This type of PDE is investigated in more detail in \cite[Section A.3]{oskar_phdthesis}. In particular, motivated by the Lorenz gauge, the interest there lies on finding a set of equations which always admit a co-closed solution. It is found in \cite[Section A.3]{oskar_phdthesis} that it is sufficient to impose the following condition:
\begin{enumerate}[label = (iii)]
    \item Applying the co-differential to the equation gives an equation directly proportional to the original equation.
\end{enumerate}
Note that this condition depends on the choice of Lorentz metric used to define the co-differential. In \cite[Lemma A.10]{oskar_phdthesis}, it is found that any co-exact wave equation that satisfies condition (iii) can be stated as 
\begin{equation}
    \extd^*\extd A+ \scalbrack{B,\extd A} + \scalbrack{\beta, \extd A} = E,
\end{equation}
where $\beta \in \Omega^1_{\mathrm{cl}}(M)$ is closed, $B\in\Omega^{2p+1}(M)$ and $E\in \Omega^p(M)$ such that 
\begin{equation*}
    \extd^* E = -\scalbrack{\beta,E}, \qquad\quad \extd^*B = -\scalbrack{\beta,B}.
\end{equation*}
Note that the string frame and Einstein frame $B$-field equation with arbitrary background field $\Bar{H}$ is of this form, cf.\ (\ref{combinedgeneinsteineqs}) and (\ref{einsteinframecombinedeinsteineqs}). (Note that in the string frame $\beta = \xi$ while in the Einstein frame $\beta = \frac{4}{d-2} \xi$.) This may shed some light on the $B$-field equation and its coupling to the metric and dilaton.

The main result for co-exact wave equations presented in \cite{oskar_phdthesis} is Theorem A.11. We shall repeat this theorem here. We use it to show that the generalised Lorenz gauge propagates well, cf.\ Proposition \ref{bfieldgaugefield_sourcefreewaveeq_prop}.
\begin{theorem}\label{coexact_wave_eqs_main_thm}
    Let $(M,g)$ be a globally hyperbolic sapcetime with spacelike Cauchy hypersurface $i\colon S\hookrightarrow M$. Let $\beta \in \Omega^1_{\mathrm{cl}}(M)$ closed, $B\in\Omega^{2p+1}(M)$ and $E\in \Omega^p(M)$ such that 
    \begin{equation*}
        \extd^* E = -\scalbrack{\beta,E}, \qquad\quad \extd^*B = -\scalbrack{\beta,B}.
    \end{equation*}
    Let $A_0 \in \Omega^p(M)$ be smooth $p$-form; the initial data. Consider the following system of PDEs:
    \begin{equation}\label{formwaveeq}
        \begin{split}
            \extd^*\extd A + \scalbrack{B, \extd A} + \scalbrack{\beta, \extd A}
            &= E, \\
            i^*\extd A &= i^*\extd A_0.\\
        \end{split}
    \end{equation}
    Herein, $N$ is the future directed normal to $S$. 
    
    If and only if $A_0$ satisfies on $S$ the constraint equation
    \begin{equation*}
        i_N\left[\extd^*\extd A_0 +  \scalbrack{B,\extd A_0} + \scalbrack{\beta, \extd A_0}- E\right] = 0 \\
    \end{equation*}
    does there exist a smooth $p$-form $A\in \Omega^p(M)$ solving the initial value problem to equation (\ref{formwaveeq}). In this case, solutions are unique up to addition of a closed $p$-form, and there exist co-closed solutions. The latter are uniquely determined by their initial value $i^*A$ on $S$.
\end{theorem}

\section{Decomposition of the (Co-)Differential of Forms over Hypersurfaces}

We collect here a few trivial results on the decomposition of forms, their differential, and their co-differential into a tangent and normal part over hypersurfaces. We work in abstract index notation to have a convenient and unambiguous way of denoting contractions. We remind ourselves that Greek indices refer to arbitrary components, while Latin indices refer to components of a parallel projection of the tensor onto a spacelike hypersurface, usually denoted by $\Sigma$.

\begin{lemma}\label{antisymmhypersurfacederlemma}
    Let $A_{\mu_1 ... \mu_k}$ be a $k$-form on a Lorentzian manifold $(M, g)$ with spacelike hypersurface $(\Sigma, g_\Sigma)$. Denote by $N = \partial_0$ a (local) unit normal on $\Sigma$, and by $\nabla$ and $D$ the LC connections on $M$ and $\Sigma$ respectively. Then, on $\Sigma$
    \begin{equation*}
        \nabla_m A_{0 m_2...m_k} = D_m A^\perp_{m_2...m_k}  - k^l_m A_{lm_2...m_k},
    \end{equation*}
    where $A^\perp = A(N) \in \Omega^{k-1}(\Sigma)$.
\end{lemma}
\begin{proof}
    \begin{equation*}
        [(\nabla A)(N)]^\parallel = [\nabla(A(N)) - A(\nabla N)]^\parallel = D(A^\perp) - A(k).
    \end{equation*}
\end{proof}

\begin{lemma} \label{extdhypersurflemma}
    Let $A$ be a $p$-form on a Lorentzian manifold $(M, g)$ with spacelike hypersurface $(\Sigma, g_\Sigma)$. Let $N$ be a (local) unit normal on $\Sigma$. Denote by $i\colon \Sigma \hookrightarrow M$ the inclusion map, and by $\extd$ and $\extd^\Sigma$ the exterior derivative on $M$ and $\Sigma$ respectively. Furthermore, denote
    \begin{equation*}
        A = A_0 = A_0^\parallel - N^\flat \wedge A_0^\perp, \qquad\quad \nabla_N A = A_1 = A_1^\parallel - N^\flat \wedge A_1^\perp
    \end{equation*}
    Then, on $\Sigma$
    \begin{equation*}
        \begin{split}
            \extd A &= \extd^\Sigma A^\parallel_0 - N^\flat \wedge \left[A_1^\parallel - \extd^\Sigma A_0^\perp + k \cdot A_0^\parallel\right], \\
            \extd^* A &= (\extd^\Sigma)^*A^\parallel_0 + A_1^\perp + A_0^\perp \tr k+ k\cdot A_0^\perp - N^\flat \wedge [-(\extd^\Sigma)^* A_0^\perp].
        \end{split}
    \end{equation*}
\end{lemma}
\begin{proof}
    Denote by $\nabla$ and $D$ the LC connections on $M$ and $\Sigma$, respectively. One,
    \begin{equation*}
        \begin{split}
            \extd A_{m_1...m_{p+1}} & = (p+1) \nabla_{[m_1} A_{m_2...m_{p+1}]} = (p+1) \partial_{[m_1} A_{m_2...m_{p+1}]}  = D_{[m_1} A_{m_2...m_{p+1}]} \\
            &= \extd^\Sigma A_{m_1...m_{p+1}}.
        \end{split}
    \end{equation*}
    Two,
    \begin{equation*}
        \begin{split}
            \extd A_{0m_1...m_p} &= (p+1) \nabla_{[0} A_{m_1...m_p]} \\
            &= \nabla_0 A_{m_1...m_p} - p \nabla_{[m_1} A_{|0|m_2...m_p]} \\
            &= \nabla_N A_{m_1...m_p} - \extd^\Sigma (i_N A)^\parallel_{m_1...m_p} -(-1)^{p}p  k^l_{[m_1} A_{m_2...m_p]l} \\
            &= (A_1^\parallel)_{m_1...m_p} - \extd^\Sigma (A_0)^\perp_{m_1...m_p} - k\cdot (A_0)_{m_1...m_p}.
        \end{split}
    \end{equation*}
    Three,
    \begin{equation*}
        \begin{split}
            \extd^* A_{m_2...m_p} &= -\nabla^\lambda A_{\lambda m_2...m_p} \\
            &= \nabla_0 A_{0m_2...m_p} - \nabla^l A_{lm_2...m_p} \\
            &= (A_1^\perp)_{m_2...m_p} + (\extd^\Sigma)^* (A_0^\parallel)_{m_2...m_p} + k^l_l A_{0m_2...m_p}\\
            &\quad + k^l_{m_2} A_{l0m_3...m_p} + ... + k^l_{m_p} A_{lm_2...m_p0} \\
            &= (A_1^\perp)_{m_2...m_p} + (\extd^\Sigma)^* (A_0^\parallel)_{m_2...m_p} - \tr k\:(A_0^\perp)_{m_2...m_p}\\
            &\quad + (-1)^{p-1} k^l_{[m_2} A^\perp_{m_3...m_p]l} \\
            &= (A_1^\perp)_{m_2...m_p} + (\extd^\Sigma)^* (A_0^\parallel)_{m_2...m_p} + \tr k\:(A_0^\perp)_{m_2...m_p}\\
            &\quad + (k \cdot A^\perp_0)_{m_2...m_p}. \\
        \end{split}
    \end{equation*}
    Four,
    \begin{equation*}
        \begin{split}
            \extd^* A_{0m_3...m_p} &= -\nabla^\lambda A_{\lambda0 m_3...m_p} \\
            &=  \nabla^l A_{0lm_3...m_p}  \\
            & = D^l A_{0lm_3...m_p} + k^{ln} A_{lnm_3...m_p} \\
            & = -(\extd^\Sigma)^* (A_0^\perp)_{m_3...m_p}.
        \end{split}
    \end{equation*}
\end{proof}

\bibliographystyle{unsrturl}  
\bibliography{literatur}
\bigskip
O.\ Schiller: oskar.schiller@uni-hamburg.de \medskip \\
Department of Mathematics, University of Hamburg,  Bundesstr.\ 55, 20146  Hamburg, Germany.
\cleardoublepage
\end{document}